\documentclass[onefignum,onetabnum]{siamonline190516}

\usepackage[utf8]{inputenc} 
\RequirePackage[OT1]{fontenc}

\usepackage{lipsum}
\usepackage{amsfonts}
\usepackage{amsmath,amssymb}
\usepackage{graphicx}
\usepackage{epstopdf}
\usepackage{algorithmic}
\ifpdf
  \DeclareGraphicsExtensions{.eps,.pdf,.png,.jpg}
\else
  \DeclareGraphicsExtensions{.eps}
\fi

\usepackage{color, latexsym}

\usepackage{esint}
\usepackage{dsfont}

\usepackage{nccmath}

{}
\def\1{1\!{\rm l}}

\newcommand{\geqa}{\gtrsim}

\newcommand{\EM}{\ensuremath}

\newcommand{\al}{\alpha}

\newcommand{\ga}{\gamma}

\newcommand{\veps}{\varepsilon}

\newcommand{\cB}{\EM{\mathcal{B}}}
\newcommand{\cC}{\EM{\mathcal{C}}}
\newcommand{\cD}{\EM{\mathcal{D}}}
\newcommand{\cE}{\EM{\mathcal{E}}}
\newcommand{\cF}{\EM{\mathcal{F}}}

\newcommand{\cL}{\EM{\mathcal{L}}}
\newcommand{\cM}{\EM{\mathcal{M}}}
\newcommand{\cN}{\EM{\mathcal{N}}}

\newcommand{\cS}{\EM{\mathcal{S}}}
\newcommand{\cT}{\EM{\mathcal{T}}}

\newcommand{\psg}{{\langle}}
\newcommand{\psd}{{\rangle}}

\definecolor{blendedblue}{rgb}{0.2,0.2,0.7}

\DeclareMathAlphabet{\mathpzc}{OT1}{pzc}{m}{it}

\newcommand{\noi}{\noindent}

\newcommand{\given}{\,|\,}

\newcommand{\rn}{\sqrt{n}}

\newcommand{\bi}{\begin{enumerate}[label=\roman*)]}
\newcommand{\ei}{\end{enumerate}}
\newcommand{\ba}{\begin{array}{rcl}}
\newcommand{\ea}{\end{array}}
\newcommand{\di}{\displaystyle}

\def\bT{\mathbb{T}}

\newcommand{\norm}[1]{\left\lVert#1\right\rVert}

\DeclareMathOperator*{\argmin}{arg\,min}

\usepackage{tikz}

\usetikzlibrary{arrows,positioning, calc}
\tikzstyle{vertex}=[draw,fill=black!15,circle,minimum size=20pt,inner sep=0pt]

\usepackage{mathtools}
\usepackage{caption}
\usepackage{subcaption}

\usepackage{enumitem}
\setlist[enumerate]{leftmargin=.5in}
\setlist[itemize]{leftmargin=.5in}

\newsiamremark{remark}{Remark}
\newsiamremark{example}{Example}
\newsiamremark{hypothesis}{Hypothesis}
\crefname{hypothesis}{Hypothesis}{Hypotheses}
\newsiamthm{claim}{Claim}
\newsiamthm{prop}{Proposition}

\headers{Optional PTs: rates and UQ}{I. Castillo and T. Randrianarisoa}

\title{Optional P\'olya trees: posterior rates and uncertainty quantification\thanks{Submitted to the editors October 2021.
\funding{This work was funded by the ANR, project ANR-17-CE40-0001 (BASICS).}}}

\author{Isma\"el Castillo\thanks{LPSM Sorbonne \& IUF 
  (\email{ismael.castillo@upmc.fr}, \url{https://www.lpsm.paris/pageperso/castillo//}).}
\and Thibault Randrianarisoa \thanks{LPSM Sorbonne 
  (\email{thibault.randrianarisoa@sorbonne-universite.fr}, \url{https://thibaultrandrianarisoa.netlify.app}).}}

\usepackage{amsopn}
\DeclareMathOperator{\diag}{diag}

\ifpdf
\hypersetup{
  pdftitle={Optional P\'olya trees: posterior rates and uncertainty quantification},
  pdfauthor={I. Castillo and T. Randrianarisoa}
}
\fi
 
\begin{document}

\maketitle

\begin{abstract}
  We consider statistical inference in the density estimation model using a tree--based Bayesian approach, with Optional P\'olya trees as prior distribution.  We derive near-optimal convergence rates for corresponding posterior distributions with respect to the supremum norm. For broad classes of H\"older--smooth densities, we show that the method automatically adapts to the unknown H\"older regularity parameter. We consider the question of uncertainty quantification by providing mathematical guarantees for credible sets from the obtained posterior distributions, leading to near--optimal uncertainty quantification for the density function, as well as related functionals such as the cumulative distribution function. The results are illustrated through a brief simulation study.
\end{abstract}



\section{Introduction}
Tree--based methods are among the most broadly used algorithms in statistics and machine learning. This goes from single tree algorithms such as CART \cite{breiman_book} or Bayesian CART \cite{bcart1, bcart2}, to the use of random forests \cite{reviewbiau, bart}, that is ensembles of trees. Due in particular to their ability to quantify uncertainty, there has been much interest in Bayesian tree--based methods.  While for frequentist methods there is a by now well--established theory in quadratic loss for CART and related algorithms, advances on the mathematical understanding of Bayesian counterparts are very recent. In \cite{rvdp20, lineroyang18}, $L^2$--posterior contraction rates are obtained for both trees and forests in a regression setting. Still in regression, the work \cite{cr21} addresses the case of the stronger supremum norm loss for Bayesian CART--type priors. The present paper can be seen as a continuation of \cite{cr21}, investigating the density estimation setting. In Bayesian density estimation, a classical tree--method is that of P\'olya trees (henceforth PTs, see e.g. \cite{gv17}, Chapter 3).  For well--chosen parameters, PTs' samples are random densities, and 
contraction rates for the corresponding posterior densities have been obtained in \cite{c17}. The idea behind P\'olya tree is to grow a fixed, infinite, tree; this is typically not flexible enough to address refined statistical goals such as adaptation. 
Notably, Wong and Ma introduced in \cite{wm10} a flexible alternative to standard PTs that they call Optional P\'olya Trees (OPTs in the sequel), which have been successfully extended and applied to a number of settings in e.g. \cite{mw11, wong16, ma17, wongnips17, cma20}. Yet, from the theoretical point of view, only posterior consistency was established in \cite{wm10} and follow-up works. Not based on (flexible) trees, we also note the different construction of spike--and--slab P\'olya trees introduced in  \cite{cm21}. 
  
There are two main goals in the present paper. The first is to continue the investigations of \cite{cr21} for tree--methods in order to obtain inference in the practically very desirable supremum norm loss, but in the model of density estimation, and the second to elaborate a theory for rates and uncertainty quantification (henceforth, UQ) for Optional P\'olya Trees. In fact, our methods enable to cover also more general priors, although for simplicity we will mostly stick to OPTs in this work. We now briefly review a number of related results. While the use of a general theory based on prior mass and testing \cite{ggv00, gv17} made a relatively broad $L^2$--theory possible \cite{lineroyang18, rvdp20}, results for the supremum norm are typically more delicate, as uniform testing rates required in \cite{ggv00} appear to be slower \cite{gn11}. Recent advances on this front include \cite{c14, hoffmannetal15, nicklray20, nauletpreprint, yoo2016}. The first supremum norm posterior rates for tree methods, optimal up to a logarithmic factor, were obtained in \cite{cr21} in regression models; we refer to \cite{cr21} for more context and references on rates for tree--based methods. 


The main results of the paper are as follows
\begin{enumerate}
\item we prove that Optional P\'olya Trees (OPTs) achieve optimal supremum--norm posterior contraction rates (up to a logarithmic factor) in density estimation: this provides an optimal rate--theory for the consistency results of \cite{wm10}, who introduced the OPT prior, for the computationally efficient case of dyadic splits.
\item we show that tree--based inference with OPTs leads to (near--) optimal uncertainty quantification in terms of confidence bands, both for the density $f$ and the distribution function $F=\int_0^\cdot f$, in an adaptive way. 
\end{enumerate}
Those constitute the first results, to the best of our knowledge, showing that tree--based methods in density estimation lead to near--optimal uncertainty quantification in terms of the supremum norm. Apart from making the consistency results of \cite{wm10} precise, this work shows that the programme for inference with tree--priors outlined in \cite{cr21}, who considered regression settings only, carries over to density estimation; the techniques presented could also be used for other tree priors beyond OPTs.

The paper is organized as follows. Section \ref{sec:prior} introduces a class of tree--based priors on density functions, of which OPTs are a special case.   Section \ref{sec:rates} states our main result on tree--based supremum norm contraction, while Section \ref{sec:uq} focuses on Uncertainty Quantification, both for the density function and smooth functionals thereof. Section \ref{sec:sims} illustrates our findings numerically through a simulation study. Section \ref{sec:disc} briefly summarises and discusses the results and future research directions. Proofs are gathered in  Section \ref{sec:proofs} and the Appendix.


\section{Dyadic tree--based random densities and Optional P\'olya trees (OPTs)}
 \label{sec:prior}

\subsection{Bayesian framework}

Adopting a Bayesian point of view, the density estimation model on $[0,1)$ consists in observing 
\begin{align}\label{def:bayesmod}
\begin{split}
 X=(X_1,\ldots,X_n)  \given f & \sim P_f^{\otimes n}\\
 f & \sim \Pi,
\end{split}
\end{align} 
where $P_f$ is the distribution on $[0,1)$ with density $f$ with respect to Lebesgue measure: $dP_f = f d\mu$, and  where  $\Pi$ is a prior distribution on densities $f$ to be defined below. The posterior distribution is then the conditional distribution of $f$ given $X$ and is denoted $\Pi[\cdot\given X]$. 

{\em Frequentist analysis of Bayesian posteriors.} To analyse mathematically the behaviour of the posterior distribution $\Pi[\cdot\given X]$, once the posterior is formed using the Bayesian model, we make the frequentist assumption that the data $X$ has actually been generated from a `true' parameter value $f_0$, that is, in the density estimation setting, $X\sim P_{f_0}^{\otimes n}$. In the sequel, we thus study the behaviour of $\Pi[\cdot\given X]$ in probability under $P_{f_0}=P_{f_0}^{\otimes n}$. For more details and context, we refer the reader to the book \cite{gv17}. 

Motivated by recent work \cite{cr21}  on Bayesian CART in regression settings (see e.g. the discussion in Section 5 of \cite{cr21}), we introduce a family of tree-based prior distributions on density functions. For simplicity, we mostly consider the case of densities on the unit interval, but our results could be extended to higher dimensions up to using slightly more complex notation, which we refrain to do here -- see, though, the discussion in Section \ref{sec:disc} for more on this --.

{\em Informal prior description.} 
 The prior on densities is defined in three steps, which will be more formally introduced below
\begin{enumerate}
\item[{\em Step 1}] a random tree $\cT$ is sampled from a prior $\Pi_\bT$ on trees;
\item[{\em Step 2}] given $\cT$, a partition $I_{\cT}$ of the unit interval is produced, built recursively in a tree fashion `along' $\cT$ with breakpoints placed at  midpoints of the successive intervals; 
\item[{\em Step 3}] given $I_{\cT}$, the output density $f$ is a histogram with random heights whose distribution follows a P\'olya tree--type law.
\end{enumerate}

\subsection{Priors $\Pi_\bT$ on full binary trees}
\begin{definition} \label{def:tree}
A {\em full binary tree} is a set of nodes $\mathcal{T}=\left\{(l,k),\ l \geq0,\ 0\leq k \leq 2^{l}-1 \right\}$ verifying the condition 
\[
(l,k)\in \mathcal{T} \implies \text{ if } l>0,\ \Big(l-1,\left\lfloor k/2 \right\rfloor\Big)\in \mathcal{T} \text{ and } \left(l,k+(-1)^k\right)\in \mathcal{T}.
\]
One then says that $\Big(l-1,\left\lfloor k/2 \right\rfloor\Big)$ is the {\em parent} node of its {\em children} $(l,k)$ and $ \left(l,k+(-1)^k\right)$, and a node with no children is called an external node or leaf; $(0,0)$ belongs to every non-empty tree and is called the tree {\em root}. We denote by $\mathcal{T}_{\text{int}}$ the set of non-terminal -- or `internal' -- nodes in $\mathcal{T}$ (i.e. those with children), and $\mathcal{T}_{\text{ext}}=\mathcal{T}\setminus \mathcal{T}_{\text{int}}$ the set of `leaves' -- also called `external' nodes --.
\end{definition}
The parent-child relationship of the pairs in a tree gives rise to the tree representation depicted on Figure \ref{fig: tree_plot_pairs}. This justifies the following terminology as we define the {\em depth} of $\cT$ as the integer \[d(\mathcal{T})\coloneqq\underset{(l,k)\in\mathcal{T}}{\max}l.\]
One further denotes by $\mathbb{T}$ the set of all binary trees and, putting a slight restriction on the maximum depth, 
\begin{equation} \label{ellmax}
\mathbb{T}_n\coloneqq\left\{\mathcal{T}\in\mathbb{T}:\ d(\mathcal{T}) \leq L_{\text{max}}\right\},\qquad
\text{with } 
L_{\text{max}}\coloneqq \Big \lfloor \log_2\left( n/\log^2(n) \right) \Big \rfloor.
\end{equation}
The prior distributions considered below put mass $1$ to the subset $\mathbb{T}_n$ of $\bT$.


\begin{figure}[!h]
\centering
     \begin{subfigure}[b]{0.45\textwidth}
        \centering
        \begin{tikzpicture}[
        very thick,
        level 1/.style={sibling distance=3cm},
        level 2/.style={sibling distance=2.5cm},
        level 3/.style={sibling distance=1cm},
        every node/.style={circle,solid, draw=black,thin, minimum size = 0.5cm},
        emph/.style={edge from parent/.style={dashed,black,thin,draw}},
        norm/.style={edge from parent/.style={solid,black,thin,draw}}
        ]
            \node [dotted] (r){$(0,0)$}
            	child {
            		node [dotted] (a) {$(1,0)$}
            		edge from parent node[left, draw=none]{}
            	}
            	child {
            		node [dotted] {$(1,1)$}
            		child {
            			node [dotted] {$(2,2)$}
            			edge from parent node[left, draw=none]{}
            		}
            		child {
            			node [dotted] {$(2,3)$}
            			edge from parent node[right, draw=none]{}
            		}
            		edge from parent node[right, draw=none]{}
            	};
        \end{tikzpicture}
        \caption{Tree pairs.}
        \label{fig: tree_plot_pairs}
     \end{subfigure}
     \hfill
     \begin{subfigure}[b]{0.45\textwidth}
         \centering
	\begin{tikzpicture}[
	very thick,
	level 1/.style={sibling distance=3cm},
	level 2/.style={sibling distance=3cm},
	level 3/.style={sibling distance=0.6cm},
	every node/.style={circle,solid, draw=black,thin, minimum size = 0.5cm},
	emph/.style={edge from parent/.style={dashed,black,thin,draw}},
	norm/.style={edge from parent/.style={solid,black,thin,draw}}
	]
		\node [rectangle] (r){$I_{00}=[0;1)$}
			child {
				node [rectangle] (a) {$I_{10}=[0;1/2)$}
				edge from parent node[left, draw=none]{}
			}
			child {
				node [rectangle] {$I_{11}=[1/2;1)$}
				child {
					node [rectangle] {$I_{22}=[1/2;1/4)$}
					edge from parent node[left, draw=none]{}
				}
    				child {
    					node [rectangle] {$I_{23}=[1/4;1)$}
    					edge from parent node[right, draw=none]{}
    				}
    				edge from parent node[right, draw=none]{}
			};
\end{tikzpicture}

         \caption{Tree partitioning $I_{\cT}$.}
         \label{fig: tree_plot_partitioning}
     \end{subfigure}
\caption{Tree $\mathcal{T}=\left\{(0,0), (1,0), (1,1), (2,2),(2,3)\right\}.$}
\label{fig: tree_plot_1}
\end{figure}
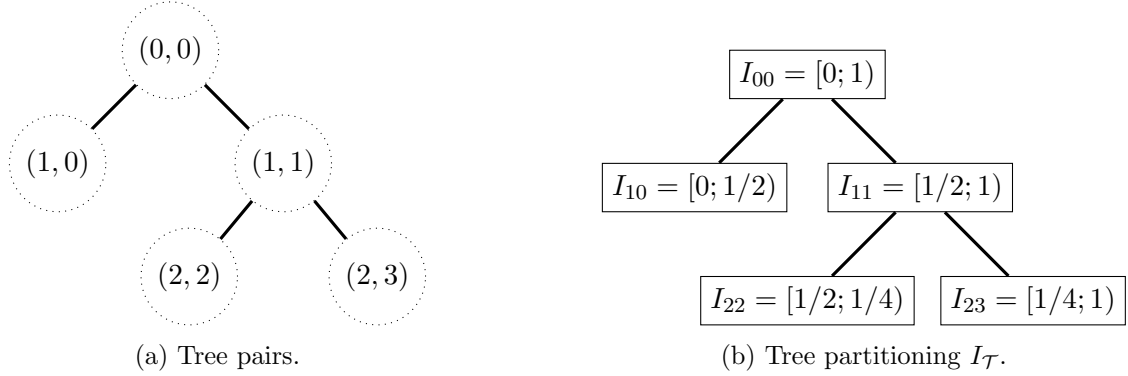

Next we give two examples of priors $\Pi_\bT$ on full binary trees. Both are actually considered in actual Bayesian CART implementations \cite{bcart1, bcart2}. 
\begin{example}[$\operatorname{GW}(p)$ Markov process on tree] \label{ex:gw}
A random tree is recursively defined by the following process. First, let us attribute to each possible pair $(l,k)$ a deterministic parameter $p_{lk}\in [0,1]$. Starting at the root node $(0,0)$, either the tree with only $(0,0)$ as node is returned with probability $1-p_{00}$, or there is a split and the tree contains not only $(0,0)$ but at least also $(1,0)$ and $(0,1)$. The construction process then continues recursively until either there are no further nodes to split, or a maximum depth $L_{max}$ is reached, after which (i.e. for $l\geq L_{max}$) we do not further grow the tree. More precisely, the recursion is from up to down ($l$ grows) and left to right ($k$ grows), as follows: given the tree contains $(l,k)$,  with probability $1-p_{lk}$ the node $(l,k)$ is a leaf; and with probability $p_{lk}$, the tree further has a split at $(l,k)$, i.e. the node $(l,k)$ has $(l+1,2k)$ and $(l+1,2k+1)$ as children in the tree.

The process producing such a random tree $\cT$ is Markov (along the complete dyadic tree) in the sense that the probability that a node $(l,k)$ further splits only depends on the fact that the node is present or not and on the parameter $p_{lk}$, but not on the rest of the tree built so far (above and to the left of $(l,k)$). By analogy to Galton--Watson processes, with here nodes having either two or zero  children with probabilities $p_{lk}$ and $1-p_{lk}$ respectively, we call  $\Pi_\bT$ as above a $\operatorname{GW}(p)$ prior, with parameters $p=(p_{lk})=(p_\epsilon)$ (we define the link between $\epsilon$  and $(l,k)$ below, in Section \ref{subsec: Partitioning}), $p_{L_{\text{max}}k}=0$.
\end{example}

\begin{example}[Conditioning on the number of leaves] \label{ex:cond}
In this construction, one samples first a number $K$ of leaves according to a prior on integers and given $K$ one then samples uniformly from the set of all full binary trees with $K$ leaves and depth at most $L_{max}$.
\end{example}

\subsection{Partitioning $I_{\cT}$}
\label{subsec: Partitioning}

Let us first introduce notation on dyadic numbers and intervals. For any binary sequence $\epsilon\in \{0,1\}^l$, its length is $|\epsilon|=l>0$. For any  dyadic number $r=k/2^l$ in $[0,1)$ with $0\leq k<2^l,\ l>0$, one writes $\epsilon(k,l)=\epsilon_1(r)\cdots\epsilon_l(r)\in \{0,1\}^l$, such that $r=\sum_{k=1}^l\epsilon_k(r)2^{-k}$, its unique decomposition in base $2^{-1}$ with $|\epsilon|=l$.  Accordingly, one introduces the dyadic intervals, for $\epsilon=\epsilon(k,l)$,
 \[ I_\epsilon\coloneqq I_{lk}\coloneqq \left[\frac{k}{2^l}, \frac{k+1}{2^l}\right), \]  
 and one sets  $I_\varnothing=I_{0,0}=[0,1)$. In addition, for any $\epsilon$ and $0< i\leq |\epsilon|$, one writes $\epsilon^{[i]}=\epsilon_1\dots\epsilon_i$. 
Also, we introduce $\mathcal{E}^*=\cup_{l=0}^\infty \left\{0;1\right\}^l$ where $\left\{0;1\right\}^0=\{\varnothing\}$.\\

To each full binary tree encoded as above as the collection of its nodes $(l,k)$, we associate a partition $I_\cT$ of the unit interval given by, with $\cT_{ext}$ the external  nodes of $\cT$ as in Definition \ref{def:tree},
\[ [0,1) = \bigcup_{(l,k)\in \cT_{ext} } I_{lk}. \]
Such a tree-based recursive partitioning of $[0,1)$ is illustrated on Figure \ref{fig: tree_plot_partitioning}. The deeper the tree locally, the more refined the corresponding partition becomes. By definition of $I_{lk}$, note that the partition has split-points  at dyadic numbers. The final partition $I_\cT$ can also be seen as being obtained from recursively splitting $[0,1)$ in halves, continuing to split locally only if the tree continues further down at that location. For this reason we talk about {\em splitting at midpoints}. Note that, still using full binary trees $\cT$, one could make splits at a different, possibly random, location. Although this makes the construction even more flexible, we shall not consider this here for simplicity (we note in passing that computationally the split--at--midpoint construction appears often to be among the easiest to simulate from, as it does not require to draw split locations; we refer to \cite{cr21}, Section 4, for more on `unbalanced' splits).

\subsection{Prior values given tree and partitioning} Once a tree $\cT$ and partitioning $I_{\cT}$ are given, we draw a random histogram over the partition given by $I_{\cT}$ by sampling heights over each sub-interval in such a way that the overall histogram is a positive density $f$ (i.e. $f>0$ and $\int_0^1 f=1$). To do so, we use a mass--splitting process along the tree $\cT$, which actually coincides with that of P\'olya trees -- we refer to the Appendix \ref{section: polya tree variants} for more on those --. This choice is for simplicity but we could consider other choices too (in this vein, the $\operatorname{Beta}(a,a)$ law at the end of Definition \ref{def-prior}  could be taken to depend on $(l,k)$ or be a different distribution).

\begin{definition}[Prior $\Pi$] \label{def-prior}
Let $\Pi_{\bT}$ be a prior on full binary trees. Let $(Y_\veps)$ be a sequence of independent variables of distribution $\operatorname{Beta}(a_{\veps0},a_{\veps1})$, for some $a_{\veps0}, a_{\veps1} \in [0,1]$, indexed by $\veps\in\cE^*$ .  The prior $\Pi$ draws a random tree--based histogram $f$ as follows
\begin{align}
\cT & \sim \Pi_{\bT} \\
f \given \cT & \sim 
 \sum_{\veps\equiv (l,k)\in \cT_{ext}} h_\veps \1_{I_{lk}}, \qquad \text{with }\
h_\veps = 2^l \prod_{i=1}^{l} Y_{\veps^{[i]}}.
\end{align}
The distribution $f\given \cT=T$ for a given $T\in\bT$ is called a $T$--P\'olya tree with parameters $(a_{\veps})$. In the sequel we set $a_{\veps}=a$  for some fixed $a>0$, in which case the distribution is denoted as $\operatorname{T--PT}(a)$. 
\end{definition}

It results from the definition that the overall prior $\Pi$ is a mixture of $T$--P\'olya trees. When the mixing distribution $\Pi_\bT$ is a $\operatorname{GW}(p)$ prior, it turns out that $\Pi$ coincides with Optional P\'olya trees introduced in \cite{wm10}, in the case of splits at midpoints. 
\begin{prop}\label{opt}
Let $\Pi$ be the mixture distribution induced on densities $f$ constructed as
\begin{align*}
\cT & \sim \operatorname{GW}(p) \\
f\given \cT & \sim \operatorname{\cT--PT}(a).
\end{align*}
Then $\Pi$ coincides with the Optional P\'olya tree of \cite{wm10} corresponding to the recursive partitioning $\left\{I_\epsilon,\ \epsilon\in\mathcal{E}^*\right\}$ with splits at midpoints and parameters $M(I_\epsilon)=\lambda(I_\epsilon)=1, K_1(I_\epsilon)=2$, stopping probabilities $\rho(I_\epsilon)=1-p_\epsilon$  for any $\epsilon\in\mathcal{E}^*$ and parameters for mass allocation $\alpha_1^1=\alpha_1^2=a$.
\end{prop}
The proof of Proposition \ref{opt} is presented in Appendix \ref{sec: link_OPT_GW}. Our notation differs slightly from \cite{wm10} (which does not make the tree connection) for two reasons: first, the tree--setting enables one to use the framework of \cite{cr21} and second, although in what follows we stick to OPTs for simplicity, the same proofs  work nearly unmodified for other tree--priors, such as the one in Example \ref{ex:cond}.

\subsection{Posterior distribution}
Let us recall that the prior $\Pi$ in Definition \ref{def-prior} is the mixture   
\begin{align}\label{def:priorgen}
\begin{split}
\cT & \sim \Pi_\bT \\                                                    
f\given \cT & \sim \Pi(\cdot\given \cT),                      
\end{split}
\end{align} 
where $\Pi(\cdot\given \cT)$ is, given $\cT$, a $\cT$--P\'olya tree. For a given dyadic interval $I$, let $N_X(I)$ denote the number of points $X_i$ that fall in $I$. The next result is proved in Appendix \ref{sec: tree_GW_posterior}.
\begin{prop}[Posterior given $\cT$] \label{prop1}
Suppose the prior is given by \eqref{def:priorgen}, where the prior given $\cT$ is a $\cT$--P\'olya tree with parameters $(a_\veps)$.  Then, in the density estimation model \eqref{def:bayesmod}, the posterior $\Pi[\cdot\given X, \cT]$ is a $\cT$--P\'olya tree with parameters $(a_\veps^X)$ given by, for any $\veps \in \cE^*$,
\[ a_\veps^X = a_\veps + N_X(I_\veps).\]
\end{prop} 
Let us now move on to describe the posterior induced on trees. We denote                                          
\begin{equation}\label{ntx}                                    
N_T(X)= \int \prod_{i=1}^n f(X_i) d\Pi(f\given \cT=T)        
\end{equation}
the marginal distribution of $X$ given $\cT=T$. It follows from Bayes' formula that $\Pi[\cdot\given X]$ induces a posterior distribution on trees given as:            
for any $T\in\bT$, and $N_T(X)$ as in \eqref{ntx},
\begin{equation} \label{post:trees}
 \Pi[\cT=T \given X] = \frac{\Pi_\bT[\cT=T] N_T(X)}{\displaystyle \sum_{T\in\bT }
\Pi_\bT[\cT=T] N_T(X)}.
\end{equation}
This is in general a fairly complicated distribution with no closed--form expression. In case the prior $\Pi_\bT$ on trees is GW$(p)$, it turns out that the posterior on trees is  GW$(p^X)$ for updated parameters $p^X$. Let, for $a>0$,
\begin{equation}\label{def:nux}
\nu_\veps^X=2^{N_X(I_{\veps})} \frac{B(a+N_X(I_{\veps0}),a+N_X(I_{\veps1}))}{
B(a,a)}.
\end{equation}
Let us now consider parameters $(p_\veps^X)$ given by the equations
\begin{equation}\label{def:rox}
\frac{p^X_{\veps}}{1-p^X_{\veps}}(1-p^X_{\veps0})(1-p^X_{\veps1})
=\frac{p_{\veps}}{1-p_{\veps}}(1-p_{\veps0})(1-p_{\veps1})\nu_\veps^X,
\end{equation}
Equations \eqref{def:rox} together admit a unique solution $(p_\veps^X)$ obtained by a bottom--up recursion noting that for $|\veps|=L_\text{max}$, $p^X_{\veps}=p_{\veps}=0$. This is verified along the proof of Proposition \ref{prop2} below.

\begin{prop}[Special case of OPTs]  \label{prop2}
In the setting of Proposition \ref{prop1}, suppose further that the distribution $\Pi_\bT$ on trees is $\operatorname{GW}(p)$  with split probabilities $(p_\veps)$. Then the posterior distribution can be described as
\begin{align*}
\Pi[\cT=\cdot\given X] & \sim \operatorname{GW}(p_\veps^X) \\
\Pi[\cdot\given X, \cT] & \sim \operatorname{\cT--PT}(a_\veps^X)
\end{align*}
with splits probabilities $(p_\veps^X)$ verifying the recursion \eqref{def:rox} and $a_\veps^X$ as in Proposition \ref{prop1}. In other words the posterior follows an OPT distribution with corresponding hyperparameters as specified in Proposition \ref{opt}.
\end{prop}
The proof of this proposition is presented in Appendix \ref{sec: tree_GW_posterior}.

\subsection{Notation and function spaces}

Below we shall consider the Hölder class of functions with support in $[0,1)$ and smoothness parameter $0<\alpha\le 1$, defined as

$$\mathcal{C}^\alpha[0,1)\coloneqq\left\{f:[0,1)\mapsto\mathbb{R},\quad \underset{x\neq y}{\text{sup}} \frac{|f(x)-f(y)|}{|x-y|^{\alpha}} <+\infty\right\}$$
and we similarly define Hölder balls with parameters $\alpha>0$ and $K\geq0$ as
$$\Sigma(\alpha, K)\coloneqq\left\{f:[0,1)\mapsto\mathbb{R},\quad \underset{x\neq y}{\text{sup}} \frac{|f(x)-f(y)|}{|x-y|^{\alpha}} \leq K\right\}.$$

{\em Bounded Lipschitz metric.} Let $(\mathcal{S},d)$ be a metric space. The bounded Lipschitz metric $\beta_{\mathcal{S}}$ on probability measures of $\mathcal{S}$ is defined as, for any $\mu,\nu$ probability measures of $\mathcal{S}$,

\begin{equation}
\beta_{\mathcal{S}}(\mu,\nu)=\sup_{F;\|F\|_{BL}\leq1}\left|\int_{\mathcal{S}}F(x)(d\mu(x)-d\nu(x))\right|,
\end{equation}

\noi where $F:\mathcal{S}\to \mathbb{R}$ and

\begin{equation}
\|F\|_{BL}=\sup_{x\in \mathcal{S}}|F(x)|+\sup_{x\neq y}\frac{|F(x)-F(y)|}{d(x,y)}.
\end{equation}

\noi This metric metrises the convergence in distribution, 
see e.g. \cite{D02}, Theorem 11.3.3.

As shown in \cite{c17}, it is also useful to introduce the Haar wavelet basis to carry out an analysis of Pólya tree-like posterior distributions. Indeed, one can relate the inclusion of a node $(l,k)$ in a tree $\mathcal{T}$ to the fact that the coefficient corresponding to the Haar wavelet function $\psi_{lk}$ in the decomposition of $f \sim \Pi[\cdot | \mathcal{T}]$  is non-zero almost surely. More precisely, the Haar basis of $L^2[0;1)$ is the family composed of the mother wavelet $\phi=\mathds{1}_{[0;1)}$ and the functions

\[\psi_{lk}({}\cdot)=2^{l/2}\psi(2^l\cdot-k)\] 
for $l\geq0$ and $0\leq k<2^l$, where $\psi=\mathds{1}_{[1/2;1)}-\mathds{1}_{[0;1/2)}$. However, as we consider the problem of density estimation, maps $f$ under scrutiny all verify $\langle f, \phi\rangle=\int_0^1 f(t)dt=1$, so that we only focus on the wavelets $\psi_{lk}$ and the corresponding coefficients $f_{lk}\coloneqq\langle f, \psi_{lk}\rangle$ in the following. As for the true density, we define $f_{0,lk}\coloneqq \langle f_0, \psi_{l,k}\rangle$.


\section{Posterior contraction rates for OPTs}
\label{sec:rates}

For any $\alpha>0$, $\mu>0$, $K\geq 0$, we define the regularity class of densities \[\mathcal{F}(\alpha,K,\mu)\coloneqq \left\{f\geq\mu,\ \ \int_0^1 f=1,\ \ f\in \Sigma(\alpha, K)\right\},\] as well as the sequence \begin{equation}\label{rate_contraction}\veps_n(\alpha) \coloneqq \left(n^{-1} \log^2 n\right)^{\frac{\alpha}{1+2\alpha}}.\end{equation}
Up to a logarithmic factor, this corresponds to the minimax supremum norm rate of estimation over the class $\mathcal{F}(\alpha,K,\mu)$, which equals $(n/\log{n})^{-\al/(1+2\al)}$ up to constants \cite{ih80}. 

\subsection{Supremum norm convergence for  the whole posterior distribution} 

We now show that the posterior distribution $\Pi[\cdot\given X]$ asymptotically concentrates most of its mass on a $\|\cdot\|_\infty$--ball of optimal radius.
 
\begin{theorem} \label{contraction_rate}
Suppose that $f_0\in\mathcal{F}(\alpha,K,\mu)$ for some $\mu>0$, $0<\alpha\leq1$ and $K\geq0$. Let $\Pi$ be an OPT prior with split probabilities $p_{lk}=\Gamma^{-l}$, $l\geq0$, $0\leq k<2^l$, $\Gamma>0$, and parameter $a>0$. Then, for $\Gamma$ large enough, any sequence $M_n\to \infty$, as $n\to\infty$, and $\veps_n=\veps_n(\alpha)$ as in \eqref{rate_contraction}, 
\[E_{f_0} \Pi\Big[ \norm{f-f_0}_\infty>M_n\varepsilon_n \given X \Big] \to 0.\]
\end{theorem}

Theorem \ref{contraction_rate} shows that an OPT posterior with split probabilities decreasing exponentially fast with nodes depth concentrates most of its mass in a supremum norm ball of (near--) minimax optimal radius, whenever the signal has regularity $\alpha\leq 1$. Some comments are in order. First, the regularity requirement $\alpha\le 1$ is typical and expected for `hard trees', which produce histogram-type estimators. An alternative would be to use `soft trees', where individual learner are smooth \cite{lineroyang18, cr21}, see also the discussion in Section \ref{sec:disc}. Second,  the slight loss of a logarithmic term in the convergence rate can be shown to be intrinsic to trees and is not due to a possible suboptimality of our rate upper--bounds: this has been formally shown in \cite{cr21}, Theorem 2, in a regression context; an analogous result could be shown in density estimation in a similar way. 

A consequence of Theorem \ref{contraction_rate} is that a posterior draw is close with high probability to the true unknown density function of interest. This settles the {\em estimation} problem, but it does not yet say much about the quantification of uncertainty, i.e. the construction of {\em confidence sets}, a question addressed in Section \ref{sec:uq}.

\subsection{Convergence rate for  the median tree} 
While Theorem \ref{contraction_rate} entails convergence in probability of a draw from $\Pi[\cdot\given X]$, one may ask what happens for aspects of such distribution, e.g. point estimators derived from it. A natural such estimator from the point of view of tree priors is the median tree estimator defined below, since there is a natural tree associated to it. Such an estimator will also turn helpful for uncertainty quantification as considered below.

The {\em median tree} is defined as the tree $\mathcal{T}^*$ whose interior nodes are 
\begin{equation}\label{mediantree}
\mathcal{T}^*_{\text{int}}=\left\{(l,k):\ \Pi[(l,k)\in \mathcal{T}_\text{int}| X]>1/2\right\},
\end{equation} 
and which is actually a tree as  defined previously (see \cite{cr21}, Lemma 13). One associates to it the {\em median tree density estimator}
\begin{equation} \label{medest} 
\hat f_{\mathcal{T}^*} =1+ \sum_{(l,k)\in \mathcal{T}^*_{\text{int}}} 2^{l/2}\frac{N_X\left(I_{(l+1)(2k+1)}\right)-N_X\left(I_{(l+1)(2k)}\right)}{n} \psi_{lk}.
\end{equation}
Lemma \ref{lemma: supnorm convergence median tree} in the appendix shows that this estimator converges in probability to the actual density $f_0$ at the same almost-minimax rate $\varepsilon_n$ in supnorm as in Theorem \ref{contraction_rate}. In Section \ref{sec:sims}, examples of $\mathcal{T}^*$ and  $\hat f_{\mathcal{T}^*}$ are presented in Figures \ref{fig: interior nodes medtree} and \ref{fig: med tree and sigma_n}.


\section{Uncertainty quantification for OPTs}
\label{sec:uq}

In nonparametrics the problem of uncertainty quantification is well--known to be more delicate than the one of estimation: first negative results to the ambitious goal of constructing confidence sets that both cover the unknown truth and have a diameter that adapts in an optimal way to the smoothness of the unknown function or density were due to \cite{li89} and \cite{low97}. The general picture that emerged in recent years following these early works is that the difficulty of the problem depends on the considered loss function and on certain testing rates of separation, see \cite{ginenickl_book}, Chapter 8.  Notably, for the supremum norm, contrary to $L^2$--losses for which some `window' of adaptation is possible, constructing adaptive confidence sets in full generality is impossible unless one restricts the set of possible functions by assuming e.g.  self--similarity conditions. Such conditions can be shown to be essentially necessary; they are also fairly natural from the practical perspective given that self--similarity is itself quite wide--spread in natural phenomena. 

Let us briefly describe the uncertainty quantification results we derive.  A first confidence band based on the posterior median and using self--similarity is built in Section \ref{cs-simple}. Next, we prove in Section \ref{cs-func} that the quantile posterior credible set for the cumulative distribution function leads to optimal UQ; this is a consequence of a more general result, an (adaptive) nonparametric Bernstein--von Mises theorem, proved in Appendix \ref{sec:npbvm}. Finally in Section \ref{cs-multi} we construct a confidence band integrating further information from some functionals that is less conservative than the simple band constructed in Section \ref{cs-simple} and achieves a target confidence level. Our results can be seen as counterparts in density estimation and for tree priors of the results in \cite{ray17}. Another approach in density estimation would be to use spike--and--slab P\'olya priors as recently considered by the second author in \cite{cm21}. Nevertheless, the latter are expected to be less efficient to compute in high--dimensions (as they, e.g., require to explore all wavelet coefficients in the different dimensions), a setting that, while not investigated in the present paper, is particularly promising for OPTs, see also the discussion in Section \ref{sec:disc}. 

\subsection{A self-similarity condition}

Here we take the same condition as in \cite{ray17} (see also \cite{ginenickl_book}). It is fairly simple to state, and can be only slightly improved (see \cite{bull}). 
\begin{definition}[Set $\cS$ of self--similar functions] \label{def-self}
Given an integer $j_0>0$ and $\al\in(0,1]$, we say that $f\in\Sigma(\alpha, K)$ is {\em self-similar} if, for some constant $\eta>0$,
\[ \norm{K_j(f)-f}_{\infty}\geq \eta 2^{-j\alpha}\ \text{for all $j\geq j_0$,}\]
where $K_j(f)=\sum_{l<j}\sum_k \langle f, \psi_{lk}\rangle \psi_{lk}$. The set of such $f$'s is denoted 
$\cS=\cS(\alpha, K, \eta)$.
\end{definition}
The condition assumes that at each resolution depth $j\ge j_0$, the overall `energy' (measured in terms of supremum norm) of the wavelet coefficients at levels larger than $j$ is lower bounded by a typical amount for $\alpha$--H\"older functions. Indeed, for any $j\ge j_0$, the quantity $\norm{K_j(f)-f}_{\infty}$ is itself also upper--bounded up to a constant by the same quantity (this follows from standard bounds on the supremum norm and the definition of the H\"older class).

\subsection{Simple confidence band} \label{cs-simple}

A first construction consists in defining a band from a centering function and a radius. A first and simple possibility consists in defining those using the median tree \eqref{mediantree}: the resulting median tree estimator \eqref{medest} can serve as center, while a radius can be defined as
\begin{equation} \label{def:sig}
 \sigma_n =v_n\sqrt{\frac{\log n}{n}} 2^{d(\mathcal{T}^*)/2},
\end{equation} 
where $d(\mathcal{T}^*)$ is the depth of the median tree $\mathcal{T}^*$, for some slowly diverging sequence  $(v_n)$ as specified below. This allows us to define the confidence band, for $\hat f_{\mathcal{T}^*}$ as in \eqref{medest},
\begin{equation}\label{confreg_def} 
\mathcal{C}_n = \left\{f:\ \norm{f-\hat f_{\mathcal{T}^*}}_\infty\leq \sigma_n\right\}.\end{equation}
Under self--similarity as in Definition \ref{def-self}, the median tree can in particular be shown to have a depth of the order of the oracle cut--off $2^{L_n^*}\approx n^{1/(2\al+1)}$ (up to a logarithmic factor, see the Appendix for a precise statement in Lemma \ref{lemma: med tree depth}) which in turn implies desirable properties for the band $\mathcal{C}_n$ as is made explicit in the next theorem.

\begin{theorem}\label{confidence_set_level}
Let $0<\alpha_1<\alpha_2\leq 1$, $K>0$, $\mu>0$ and $\eta>0$. Let $\Pi$ be the same prior as in Theorem \ref{contraction_rate}, $\cC_n$ as in \eqref{confreg_def} with $v_n/\log^{1/2} n\to\infty$, then uniformly on $f_0\in\cS(\alpha, K, \eta)\cap \mathcal{F}(\alpha,K,\mu),\ \alpha\in[\alpha_1,\alpha_2]$,
\[ \left|\mathcal{C}_n\right|_\infty =O_{P_{0}}\left(v_n\Bigg(\frac{\log n}{n}\Bigg)^{\alpha/(2\alpha+1)}\right)\]
and
\[P_0\left[f_0\in \mathcal{C}_n\right]=1+o(1),\qquad \Pi[\cC_n\given X]=1+o_{P_0}(1).\]
\end{theorem}
For a slowly diverging sequence $(v_n)$, the diameter of $\mathcal{C}_n$ is then within a logarithmic factor of the minimax rate of estimation on $\Sigma(\alpha,K)$ with high probability. It is attained adaptively (the definition of $\mathcal{C}_n$ does not depend on $\alpha$) for any window $[\alpha_1;\alpha_2]$. The set $\cC_n$ allows to quantify uncertainty on $f_0$ as it is  an asymptotic confidence set, and it is also a credible set of credibility going to $1$.

\subsection{UQ for functionals: a Donsker--type theorem} \label{cs-func}
$\ $\\

{\em OPTs with flat initialisation.} Let us introduce a slight modification of the OPT prior where trees from the prior distribution are constrained to include all nodes of depth less than some  number $l_0=l_0(n)$, slowly diverging to $\infty$.  

\begin{definition}
A prior on densities $\Pi$ of the type \eqref{def:priorgen} is said to have {\em flat initialisation up to level} $l_0=l_0(n)$ if the prior  
on trees $\Pi_\bT$ verifies
\[ \Pi_\bT\left[\bigcap_{l\le l_0(n), k} \{(l,k)\in \cT\} \right] = 1. \]
\end{definition}


The next result considers the behaviour of the induced posterior on $F(\cdot)=\int_0^\cdot f$, that is on the distribution function for an OPT prior on $f$.  Let us also define, for $\hat f_{\mathcal{T}^*}$ the median tree estimator,
\begin{equation} \label{def-cdfmed}
\hat F_n^{med}(t)=\int_0^t \hat f_{\mathcal{T}^*}(u)du.
\end{equation}
Let us recall that for $Q$ a probability measure on $[0,1]$ of distribution function $H$,  a $Q$--Brownian bridge is a centered Gaussian process $Z(t)$ with covariance function $E[Z(s)Z(t)]=\min(H(s),H(t))-H(s)H(t)$ and $0\le s,t\le 1$.  
\begin{theorem}[Donsker's theorem for OPTs] \label{thm-Donsker}
Let $X=(X_1,\ldots,X_n)$ be i.i.d. from law $P_0$ with density $f_0$.  
Let $f_0\in\mathcal{F}(\alpha,K,\mu)$, for some $\alpha\in(0;1]$, $K\geq0$, $\mu>0$.  Let $\Pi$ be an OPT prior with flat initialisation up to level $l_0(n)$ that verifies $\sqrt{\log{n}}\le l_0(n)\le \log{n}/\log\log{n}$, and other than that for $l>l_0(n)$ with same parameters as the prior in Theorem \ref{contraction_rate}. 

Let $G_{P_0}$ be a $P_0$-Brownian bridge $G_{P_0}(t), t\in[0,1)$. For $\hat F_n^{med}$ as in \eqref{def-cdfmed}, as $n\to\infty$, 
\[ \beta_{C[0,1)}\left(\cL(\rn(F-\hat F_n^{med})\given X),\cL(G_{P_0})\right) \to^{P_{f_0}} 0. \]
Furthermore, for $F_n$ the empirical distribution function,  as $n\to\infty$, 
\[ \beta_{L^\infty[0,1)}\left(\cL(\rn(F- F_n)\given X),\cL(G_{P_0})\right) \to^{P_{f_0}} 0. \]
\end{theorem}

This implies that the induced posterior distribution $\cL(\rn\|F-\hat F_n^{med}\|_\infty \given X)$ converges weakly in probability to $\cL(\|G_{P_0}\|_\infty)$. Furthermore, for $0<\ga<1$, the credible set 
\[ \cF_n=\{F:\ \|F-\hat F_n^{med}\|_\infty \le \rho_n^X\},\]
with $\rho_n^X$ chosen such that $\Pi[\cF_n\given X]=1-\gamma$, 
is an asymptotically  optimal (efficient) confidence set of level $1-\ga$.  
We refer to \cite{cn14} for more details on this; note that in the latter paper the results are for priors of fixed regularity only, whereas here the prior additionally enables adaptation to the smoothness of $f$. The behaviour of the credible set $\cF_n$ is illustrated in Figure \ref{fig: cdf cs}.

\subsection{Multiscale confidence band} \label{cs-multi}
Here we follow the approach introduced in \cite{cn13, cn14} and first briefly recall the idea. One wishes to define a `multiscale' space (i.e. defined from wavelet coefficients) with an associated metric that is weak enough so that convergence of the posterior distribution for $f$ in that space converges at rate $1/\sqrt{n}$, instead of the slower nonparametric rate  of order $n^{-\al/(2\al+1)}$. In such space one can then formulate a convergence of the posterior to a Gaussian limit, namely a nonparametric Bernstein--von Mises theorem.  Below we only define the multiscale space as it is used in the definition of the credible band and postpone details on the precise statement of convergence to Appendix \ref{sec:npbvm}.

Let us call the sequence $w=(w_l)_{l\geq 0}$ `admissible' if $w_l/\sqrt{l} \to \infty$ as $l \to \infty$. For such a sequence, let us define
\begin{equation}\label{aimezero}
\mathcal{M}_0=\mathcal{M}_0(w)=\left\lbrace x=(x_{lk})_{l,k},\ \ \lim_{l \to \infty} \max_{0\leq k < 2^l} \frac{|x_{lk}|}{w_l}=0\right\rbrace.
\end{equation}
\noindent Equipped with the norm $\di \|x\|_{\mathcal{M}_0}= \sup_{l\geq 0}\max_{0\leq k < 2^l} |x_{lk}|/w_l$, this is a separable Banach space \cite{cn14}. In a slight abuse of notation, we  write $f \in \mathcal{M}_0$ if the sequence of its Haar wavelet coefficients $(\langle f, \psi_{lk}\rangle)_{l,k}$ belongs to that space.

Let us consider a credible ball in the space $\cM_0$: recalling the definition \eqref{medest} of the median tree estimator $\hat f_{\mathcal{T}^*}$, let us choose $R_n=R_n(X)$ in such a way that
\begin{equation} \label{def:rn}
\Pi[\| f- \hat f_{\mathcal{T}^*} \|_{\cM_0(w)}\le R_n/\sqrt{n} \given X] = 1-\ga
\end{equation}
(or possibly $\ge 1-\ga$ if the equation has no exact solution, in which case the limit in the confidence statement of the next proposition is replaced by a liminf and equality by $\ge$). 

Let us define, for $R_n$ as in \eqref{def:rn}, $\sigma_n$ as in \eqref{def:sig} and $f_{\mathcal{T}^*}$ the median tree estimator \eqref{medest},
\begin{equation} \label{cred:multi}
 \cC_n^\cM = \left\{f:\ \| f-\hat f_{\mathcal{T}^*}\|_\infty\leq \sigma_n\right\}\ \bigcap\ \left\{ f:\ \|f-\hat f_{\mathcal{T}^*}\|_{\mathcal{M}_0(w)}\leq R_n/\sqrt{n}\right\}.
 \end{equation}

The next result states that $\cC_n^\cM$ is under self--similarity asympotically a confidence
band of prescribed level $1-\ga$. 

\begin{proposition}\label{prop:cmulti}
Let $0<\alpha_1<\alpha_2\leq 1$, $K>0$, $\mu>0$ and $\eta>0$. 
Let $\cC_n^\cM$ be defined by \eqref{cred:multi},  for $v_n/\log^{1/2} n\to\infty$, and $\Pi$ an OPT prior with flat initialisation up to level $l_0(n)$ that verifies $\sqrt{\log{n}}\le l_0(n)\le \log{n}/\log\log{n}$, and other than that for $l>l_0(n)$ with same parameters as the prior in Theorem \ref{contraction_rate}. First, the set $\cC_n^\cM$ is a $(1-\ga)-$credible band as, uniformly on $\alpha \in[\alpha_1,\alpha_2]$ and $f_0\in\cS(\alpha, K, \eta)\cap\cF(\alpha,K,\mu)$, 
 \[ \Pi[\cC_n^\cM \given X] = 1-\ga+ o_{P_0}(1). \]
Further, under the same conditions,
\begin{align*}
 \left|\cC_n^\cM\right|_\infty & =O_{P_{0}}\left(v_n\Bigg(\frac{\log n}{n}\Bigg)^{\alpha/(2\alpha+1)}\right), \\ 
 P_0\left[f_0\in \cC_n^\cM \right] &= 1-\ga+o(1).
 \end{align*}
\end{proposition}

Proposition \ref{prop:cmulti} quite directly follows from combining Theorem \ref{confidence_set_level}, which concerns $\cC_n$ and the nonparametric BvM Theorem \ref{BVM} proved in the Appendix, which concerns the second part of the intersection in \eqref{cred:multi}. Compared to $\cC_n$ the advantage of $\cC_n^\cM$ is that it uses more `posterior information' by intersecting with the $\cM_0(w)$ credible ball, resulting in a credible ball with both credibility and confidence close to a given user--specified confidence level $1-\gamma$. By contrast, $\cC_n$ was more `conservative' in this respect, having credibility and confidence both going to $1$. The behaviour of the credible band $\cC_n^\cM$, in particular in comparison to $\cC_n$ from \eqref{confreg_def}, is illustrated in simulations in the following Section \ref{sec:sims}.


\section{Simulation study} \label{sec:sims}

We consider the credible sets $\cC_n$ and $\cC_n^\cM$ defined in \eqref{confreg_def} and \eqref{cred:multi} respectively and illustrate their coverage and diameter properties numerically through a simulated study.

We focus on a prior as in Proposition \ref{prop:cmulti}, with parameters $\Gamma=1.1$, $a=1$ and $l_0(n)=\sqrt{\log{n}}$. We take four fairly different densities $f_0$, illustrating different aspects of inference and UQ with Optional P\'olya trees:
\begin{itemize}
\item The triangular density $x\mapsto (.5+2*x)1_{0\leq x<0.5} + (1.5-2*(x-.5))1_{0.5\leq x<1}$ that is Lipschitz regular.
\item The density $$t\mapsto\frac{e^{W_t}}{\int_0^1e^{W_s}ds}$$ where $(W_t)_{t\in[0;1)}$ is a Brownian motion that is almost surely $(1/2-\delta)$--Hölder regular for  any $0<\delta<1/2$.
\item The density $$t\mapsto C\left(e^{W_t}1_{0\leq t<0.5} + c1_{0.5\leq x<1}\right\}$$ for $(W_t)_{t\in[0;1)}$ a Brownian motion and $C,c$ real numbers such that this actually defines a continuous density function. In this case, the regularity is different and of a higher order on the second half of the interval.
\item The sine density $t\mapsto 1+0.5*\sin(2\pi x)\in C^{\infty}([0;1))$.
\end{itemize}

\begin{figure}[h!]
\center
  \caption{Interior nodes $\cT_{\text{int}}^*$ of the median tree - $n=10^5$.}
  \includegraphics[width=\textwidth]{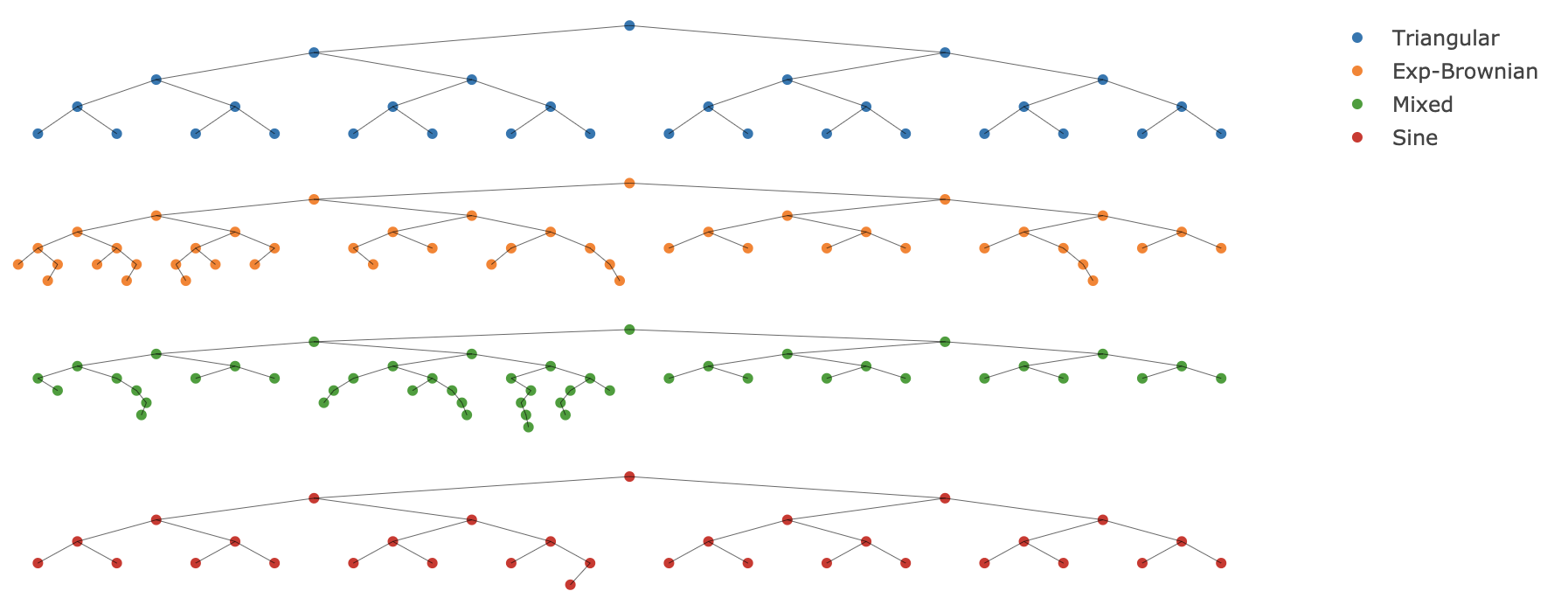}
  \label{fig: interior nodes medtree}
\end{figure}

We first illustrate the  behaviour of the median tree $\cT^*$ and the associated estimator $\hat f_{\mathcal{T}^*}$ defined in \eqref{medest} in these different situations. In Figure \ref{fig: interior nodes medtree}, we observe how this tree adapts to the regularity of the underlying sampling density $f_0$ via the interior nodes it selects. First, in the case of the smoother sine and triangular densities, fewer nodes are included, while the tree grows deeper with the other two more irregular signals. Indeed, as mentioned before and explicited in Lemma \ref{lemma: med tree depth}, the median tree can be shown to have a depth close to the oracle cut-off $L_n^*$, satisfying $2^{L_n^*}\approx n^{1/(2\al+1)}$. However, although the sine density is even more regular than the triangular one, their respective median trees have a similar behaviour and grow at the same pace. Indeed, since we use a piecewise constant tree estimator which relates to the Haar wavelet basis, our method cannot leverage additional regularities, beyond $\mathcal{C}^1[0,1)$. Finally, when it comes to the mixed density, the median tree has a spatial-dependent behaviour. It includes much more nodes in regions that corresponds to the first half of the sampling space, where the target regularity is that of the exp-Brownian density. As for the other half of the sampling space, it doesn't get deeper than $l_0(n)$. It highlights a desirable feature of tree-based methods, that is their spatial adaptivity. While we consider adaptation to global regularity in our theoretical results, one could also consider local adaptation, as was recently considered in \cite{RR21}, where  results on local adaptation for tree--based priors (among others) are obtained in a regression setting. 

In Figure \ref{fig: med tree and sigma_n}, for the four sampling densities, we illustrate the estimator $\hat f_{\mathcal{T}^*}$ (orange) and the bounds of the credible set $\mathcal{C}_n$ (red), where we took $v_n=(\log n)^{0.501}$ in \eqref{def:sig}. The estimator \eqref{medest} struggles to approximate the 'spiky' portions of the most irregular signals. Still, in any case, the credible band covers the true density $f_0$ as expected.

\begin{figure}[h!]
\center
  \caption{Median tree estimator $\hat f_{\mathcal{T}^*}$ and credible set $\mathcal{C}_n$ - $n=10^4$}
  \includegraphics[width=\textwidth]{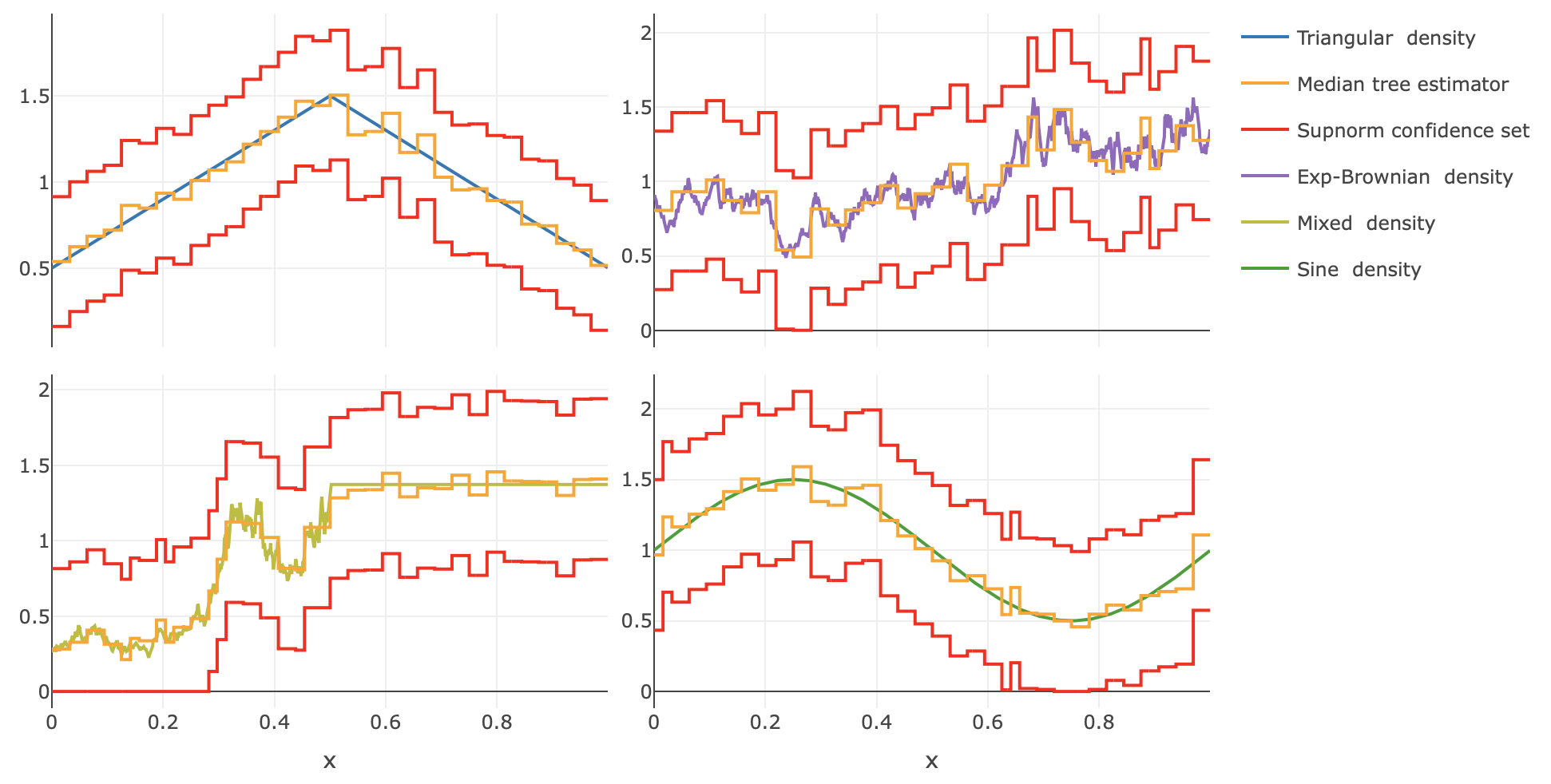}
  \label{fig: med tree and sigma_n}
\end{figure}

Then, to illustrate the intersected set $\cC_n^\cM$, defined in \eqref{cred:multi} via a multiscale condition, we sampled $10000$ draws from the posterior and plotted, in Figure \ref{fig: confidence band multiscale}, $100$ of those belonging to the confidence band (blue), for $\gamma=0.05$. Most of those samples do not seem to lie close to the bounds of  $\mathcal{C}_n$ which is consistent with the fact that $\mathcal{C}_n$, resp. $\cC_n^\cM$, has a posterior mass close to $1$, respectively $0.95$. Though our illustrations concern the intersection of $\cC_n^\cM$ with the support the posterior, via the representation of posterior draws, it appears that $\cC_n^\cM$ is actually smaller than $\mathcal{C}_n$.

\begin{figure}[h!]
\center
  \caption{Posterior sample in the confidence band $\cC_n^\cM$ - $\gamma=0.05$ and $n=10^4$.}
  \includegraphics[width=\textwidth]{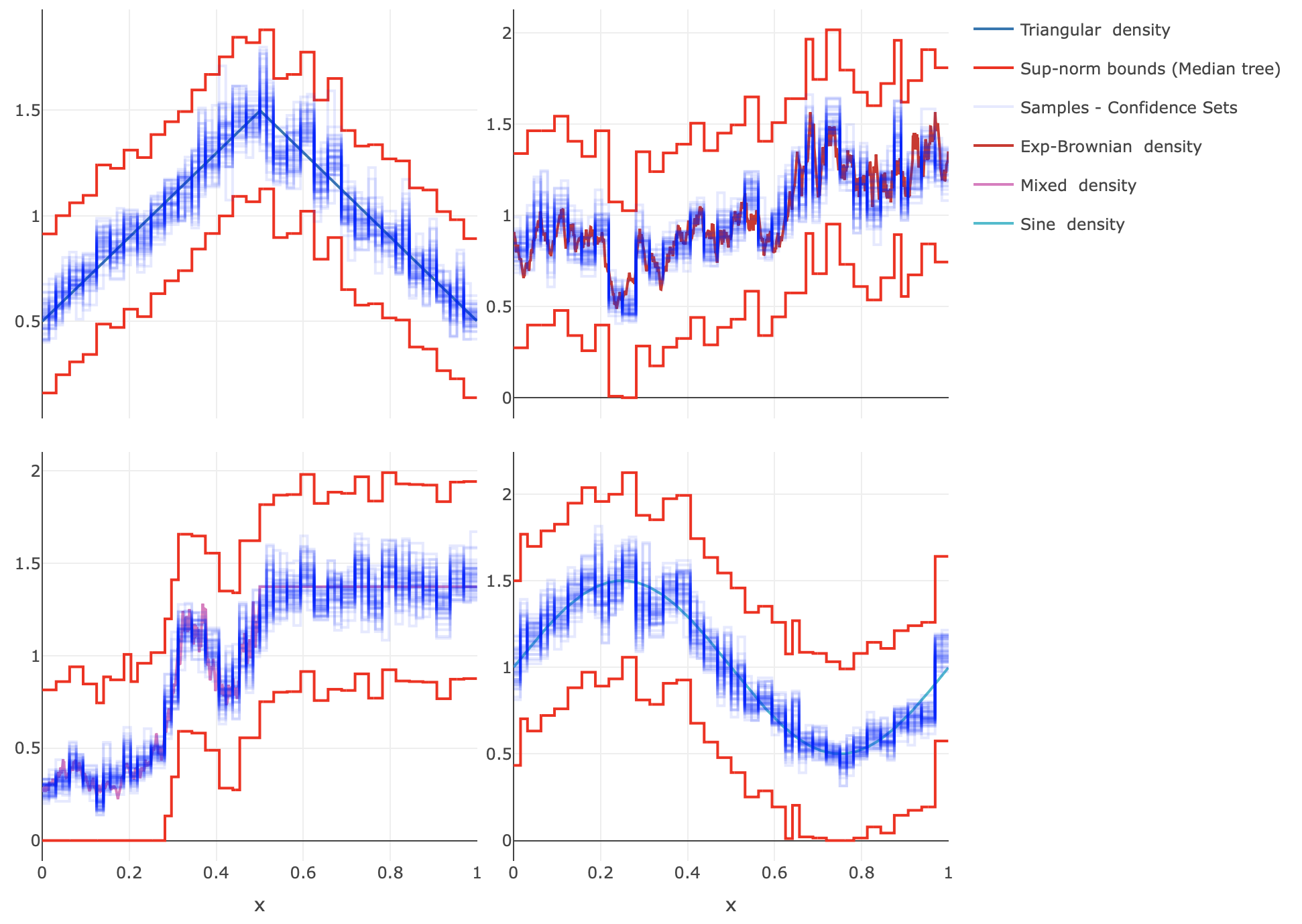}
  \label{fig: confidence band multiscale}
\end{figure}

As for the confidence sets  $\cF_n$ on the cumulative distribution function $F_0(\cdot)=\int_0^{\cdot} f_0(t)dt$, we illustrate an example in Figure \ref{fig: cdf cs} for a smaller sample size of $n=10^3$ and $\gamma=0.95$. The bounds of $\cF_n$ follow tightly the true signal and the set covers it, in spite of the fewer number of observations available compared to previous plots. Indeed, following the discussion after Theorem \ref{thm-Donsker}, $\cF_n$ has a radius decreasing at the parametric rate $\sqrt{n}^{-1}$.

\begin{figure}[h!]
\center
  \caption{Posterior samples in the confidence set $\cF_n$ - $n=10^3$.}
  \includegraphics[scale=0.35]{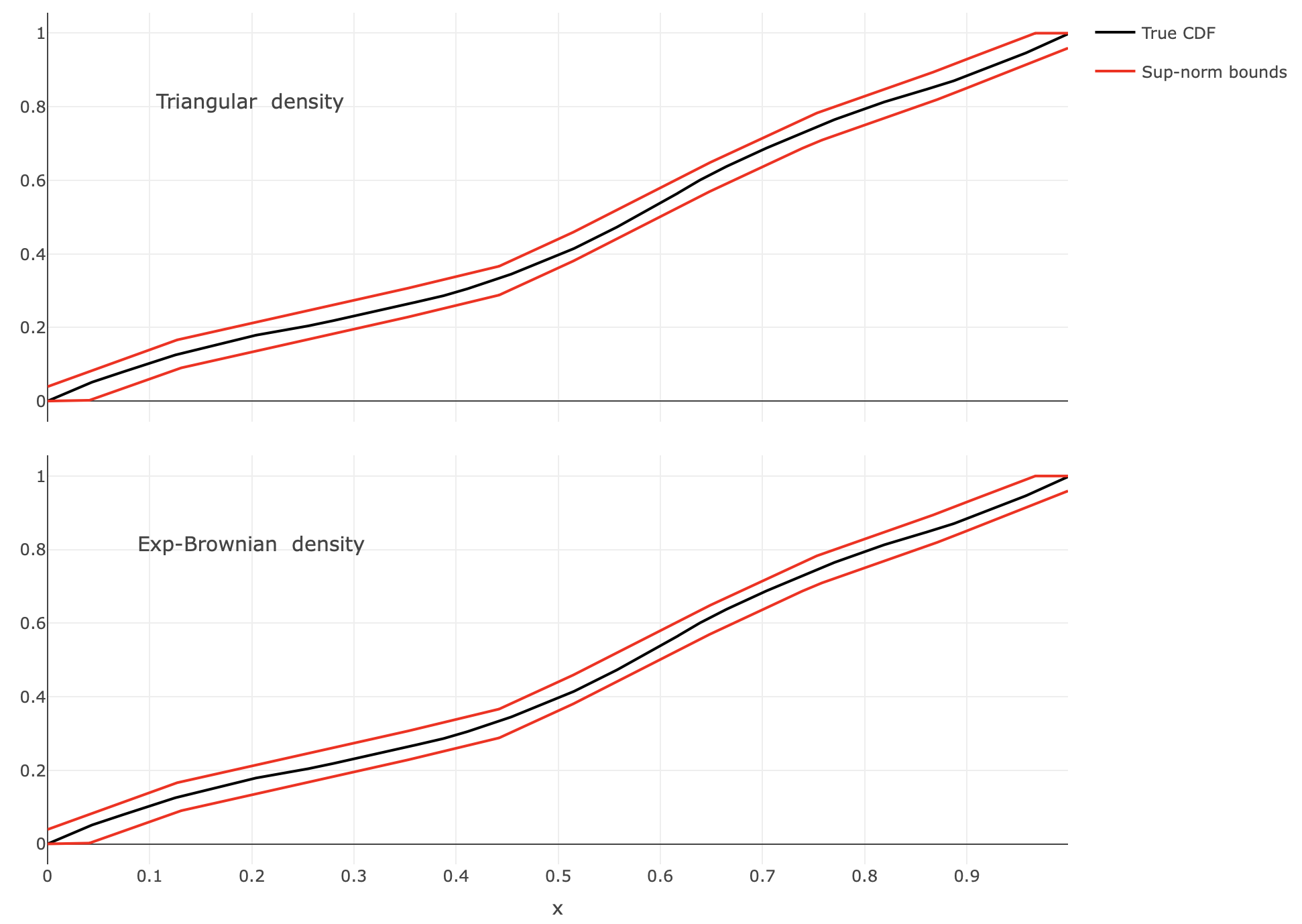}
  \label{fig: cdf cs}
\end{figure}

\begin{table}[]
\footnotesize
\centering
\begin{tabular}{l|l|l|l|l|}
\hline
\multicolumn{1}{|l|}{Chosen significance $\gamma$}                                     & $0.99$   & $0.95$   & $0.9$     & $0.85$     \\ \hline
                                                                                       & \multicolumn{4}{c|}{$n=10^4$}                \\ \hline
\multicolumn{1}{|l|}{Credibility of $C_n^{L^{\infty}}$}                                & $0.99$   & $0.95$   & $0.9$     & $0.85$     \\ \hline
\multicolumn{1}{|l|}{Credibility of $\mathcal{C}_n^{\mathcal{M}}$}                     & $0.99$  & $0.95$  & $0.8981$  & $0.85$   \\ \hline
\multicolumn{1}{|l|}{Credibility of $C_n^{L^{\infty}}\cap\mathcal{C}_n^{\mathcal{M}}$} & $0.9801$ & $0.9029$ & $0.8108$  & $0.725$   \\ \hline
\multicolumn{1}{|l|}{Credibility of the intersection if independence}                  & $0.9801$ & $0.9025$ & $0.81$  & $0.7225$   \\ \hline
                                                                                       & \multicolumn{4}{c|}{$n=10^5$}                \\ \hline
\multicolumn{1}{|l|}{Credibility of $C_n^{L^{\infty}}$}                                & $0.99$   & $0.95$   & $0.9$     & $0.85$     \\ \hline
\multicolumn{1}{|l|}{Credibility of $\mathcal{C}_n^{\mathcal{M}}$}                     & $0.9894$ & $0.9494$ & $0.8994$  & $0.8494$   \\ \hline
\multicolumn{1}{|l|}{Credibility of $C_n^{L^{\infty}}\cap\mathcal{C}_n^{\mathcal{M}}$} & $0.9801$ & $0.9028$ & $0.8118$  & $0.7254$   \\ \hline
\multicolumn{1}{|l|}{Credibility of the intersection if independence}                  & $0.9795$ & $0.9019$ & $0.8095$ & $0.722$ \\ \hline
\end{tabular}
\caption{Credibility of sets $C_n^{L^{\infty}}$ and $\mathcal{C}_n^{\mathcal{M}}$ for the triangular density $f_0$.}
\label{table: credibility of sets}
\end{table}

We end this section with an illustration of a phenomenon that was noticed and established in \cite{ray17} for a spike-and-slab prior in a regression setting. Namely, since we constructed an adaptive $(1-\gamma)$-confidence bands whose diameter in supnorm shrinks at an almost optimal rate, one may wonder how much it differs from the $(1-\gamma)$-credible band in the supremum norm $C_n^{L^{\infty}}\coloneqq\left\{f:\ \norm{f-\hat f_{\mathcal{T}^*}}_{\infty}\leq Q_n(\gamma)\right\}$, where $Q_n(\gamma)$ is chosen such that $\Pi\left[C_n^{L^{\infty}} | X\right]\geq 1-\gamma$. In a white noise regression setting,   \cite{ray17} proved that these two sets are asymptotically independent (see Theorem 5.3 therein), in the sense that $\Pi\left[C_n^{L^{\infty}}\cap\mathcal{C}_n^{\mathcal{M}} | X\right]\overset{P_{f_0}}{\to}(1-\gamma)^2$. As above, we sampled $10^4$ draws from the posterior to estimate de posterior credibility of the different sets, which we present in Table \ref{table: credibility of sets}. The results seem to indicate that the independence phenomenon of the credible sets as described above still hold in the present density estimation setting, as the margin of difference observed is of the order of the Monte-Carlo error. Intuitively speaking, this independence under the posterior if true (at least asymptotically) would mean that the two credible sets reflect {\em different} aspects of the posterior distribution. Although this result from \cite{ray17} is seemingly verified in density estimation with an OPT prior, we did not investigate this question from a theoretical point of view in the present paper; we expect the proof to be significantly more involved than in the (conjugate) Gaussian white noise setting and we leave this point for future work.

\section{Discussion}
\label{sec:disc}

In the present work we establish an inference theory for Optional P\'olya trees introduced in \cite{wm10} by deriving posterior contraction rates as well as confidence bands for the problem of uncertainty quantification. By contrast, only posterior consistency had been previously obtained until now for such priors. Although we focus on this class of prior distributions, we point out that our proofs and results also apply to different tree priors, such as ones conditioning on the number of leaves as in Example \ref{ex:cond}. The results and proofs highlight how beneficial a multiscale approach to study tree-based methods, as introduced in \cite{cr21}, can be.

As for related priors in density estimation, non-adaptive contraction rates were obtained in \cite{c17} for Pólya trees for carefully chosen regularity-dependent parameters of the Beta random variables. The addition of a hyperprior on the tree structure in OPTs allows for adaptation, so that the Beta parameters can be set as an arbitrary constant (a similar comment can be done about Spike-and-slab Pólya trees \cite{cm21}). The Beta variables in the Pólya-like mass allocation mechanism could be replaced by another distribution, but we sticked to them for simplicity of analysis and presentation.

In Section \ref{sec:sims}, we mentioned some further results on OPTs to be investigated. First, tree-based methods have a natural ability to adapt to the local regularity. While this has been proved in \cite{RR21} in a regression setting, this should also be the case with OPTs in density estimation. Another expected advantage of trees is that in high-dimensional settings, they induce a `tree--structured' sparsity, which could help in addressing the curse of dimensionality. As original OPTs \cite{mw11} have been introduced in arbitrary dimensions, it is natural to further our theoretical analysis in this direction, Also, the interesting alleged posterior independence of sets $C_n^{L^{\infty}}$ and $\mathcal{C}_n^{\mathcal{M}}$ still needs to be proven and would confirm that the two constructions rely on somewhat different aspects of the posterior distribution in density estimation too.

Finally, the almost-optimal rates we obtain are valid for H\"older regularities up to $1$. This is related to the fact that samples of OPTs are piecewise constant on some random partition. In order to achieve faster rates for smoother densities, one possibility explored in \cite{lineroyang18}   consists in replacing  `hard' (histogram) trees with `smooth' trees. 
Another promising possibility is to look at forests priors. Indeed, the aggregation of many trees tends to result in estimators that are more `regular' and thereby more suitable to the estimation of smoother objects: for frequentist estimators in regression, this was noted in \cite{ArlotGenuer14, mourtada2018minimax} for regularities $\alpha\le 2$. The recent work \cite{TR20} establishes that, when an $L^1$ or Hellinger loss on densities is considered,  forests of P\'olya trees enable adaptation to arbitrary regularities $\alpha$. This will be investigated elsewhere.


\section{Proof of the main results} \label{sec:proofs}
Below, the depth $L_n=L_n(\alpha)$ defined as
 \begin{equation}\label{cutoff}2^{L_n} = c_0 (n/\log n)^{\frac{1}{1+2\alpha}},\end{equation} for some $c_0>0$, will be helpful in our theoretical analysis.
Also, $C$ stands for a generic constant whose precise value we do not track and can change from line to line.

\subsection{Proof of Theorem \ref{contraction_rate}}

Let's write $T_n = \{\mathcal{T} \given d\left(\mathcal{T}\right) \leq L_n,\  S(f_0, \tau) \subset \mathcal{T}   \}$, $S(f_0, \tau)$ as in Lemma \ref{huge signals taken}, and $\mathcal{E}_n = \{f:\ \exists \mathcal{T}\in T_n,\ f \text{ piecewise constant on } I_{\cT}  \}$. Moreover, we write, for $L_n$ as in \eqref{cutoff} and any tree $\mathcal{T}\in \mathbb{T}_n$, the following othogonal projections of $f_0$: $f_0^{\mathcal{T}}$ onto the span $\{\psi_{lk}\given (l,k)\in \mathcal{T}\}$, $f_0^{L_n^c}$ onto $\{\psi_{lk}\given l>L_n\}$, and $f_0^{\mathcal{T} ^c, L_n}$ onto the orthocomplement of the union of the two last spans.  For $f_0\in\mathcal{C}^\alpha[0,1)$, $0<\alpha\leq 1$, we have in particular that
\begin{equation}\label{eq: bias_L_n} ||f_0^{L_n^c}||_\infty\leq \sum_{l>L_n}2^{l/2} \underset{0\leq k<2^l}{\max} |f_{0,lk}| \lesssim \sum_{l>L_n} 2^{-l\alpha} \lesssim \left(n^{-1} \log n\right)^{\frac{\alpha}{1+2\alpha}}, \end{equation} (see for instance \cite{c17}).
Then, for any density $f_0$, we have the upper bound, for $\mathcal{B}_M$ as in Lemma \ref{obs_event},


\begin{align*}
&  \Pi\left[ \| f-f_0\|_\infty>M_n\epsilon_n \given X \right] \\
& \le 
\Pi[\mathcal{E}_n^c \given X]\mathds{1}_{\mathcal{B}_M}  + \Pi\left[ \norm{f-f_0}_\infty>M_n\epsilon_n,\  f\in \mathcal{E}_n \given X \right] 
 \mathds{1}_{\mathcal{B}_M} + \mathds{1}_{\mathcal{B}_M^c}.
\end{align*}
On one hand, Lemma \ref{obs_event} guarantees that $\mathbb{P}_0\left(\mathcal{B}_M^c\right) =o(1)$ for $M$ large enough and Lemmas \ref{not too deep} and \ref{huge signals taken} ensures that 
\[E_{f_0}\left\{\Pi_f[\mathcal{E}_n^c \given X]\mathds{1}_{\mathcal{B}_M}\right\}=o(1).\]
On the other hand, we also have the inequality $\norm{f-f_0}_\infty \leq \norm{f-f_0^{\mathcal{T}}}_\infty + \norm{f_0^{\mathcal{T} ^c, L_n}}_\infty + \norm{f_0^{L_n^c}}_\infty$. This allows us to control the last term in the above upper bound by mean of the Markov inequality:
\begin{align}
\label{Markov}
\begin{split}
\Pi\Big[ & f\in \mathcal{E}_n,\quad\norm{f-f_0}_\infty>M_n\epsilon_n \given X \Big]\mathds{1}_{\mathcal{B}_M}
 \le (M_n\epsilon_n)^{-1} \int_{\mathcal{E}_n} \norm{f-f_0}_\infty d\Pi[f, \mathcal{T} \given X] \mathds{1}_{\mathcal{B}_M}\\
&\le  (M_n\epsilon_n)^{-1}\Big[ \int_{\mathcal{E}_n} \norm{f-f_0^{\mathcal{T}}}_\infty d\Pi[f, \mathcal{T} \given X] \mathds{1}_{\mathcal{B}_M}  \\
& \qquad \qquad +\int_{\mathcal{E}_n} \norm{f_0^{\mathcal{T} ^c, L_n}}_\infty d\Pi[\mathcal{T} \given X] \mathds{1}_{\mathcal{B}_M} +\norm{f_0^{L_n^c}}_\infty \Big],
\end{split}
\end{align}
and \eqref{eq: bias_L_n} ensures that the last term above is $o(1)$. Similarly, for the second term, using the definition of $\mathcal{E}_n$ and denoting $L^*$ the largest integer such that $2^{-L^*(\alpha+1/2)}\geq n^{-1/2} \log n$, 

\begin{align*}
||f_0^{\mathcal{T} ^c, L_n}||_\infty &\leq \sum_{l\leq L_n}2^{l/2} \underset{k: (l,k)\not\in\mathcal{T}}{\max} |f_{0,lk}| \lesssim \sum_{l\leq L_n}2^{l/2} \left( \underset{0\leq k<2^l}{\max} |f_{0,lk}| \wedge \log n\ /\ \sqrt{n} \right)\\
&\lesssim \sum_{l\leq L^*} 2^{l/2} \frac{\log n}{\sqrt{n}} + \sum_{L^*< l\leq L_n} 2^{l/2}2^{-l(1/2+\alpha)}\lesssim 2^{L^*/2}\frac{\log n}{\sqrt{n}} + 2^{-L^*\alpha}\lesssim 2^{-L^*\alpha}.
\end{align*}
This allows us to conclude that the second term in the bound \eqref{Markov} is also of the order $o(1)$. It remains to bound the first term in the bound that is also of order $o(1)$ according to Lemma \ref{no distortion}. This concludes our proof. \\

It remains to prove the different lemmas we used to upper bound the different terms above.

\begin{lemma}
\label{not too deep}
Suppose  $f_0\in\mathcal{F}(\alpha,K,\mu)$, for some $\mu>0$, $0<\alpha\leq1$, $K>0$, and assume $f$ follows a prior as in Theorem \ref{contraction_rate}. Then, for any $M>0$ as in Lemma \ref{obs_event} and $\Gamma$ large enough, on events $\mathcal{B}_M$, we have, as $n\to\infty$,
\[\Pi[d(\mathcal{T}) > L_n \given X] \to 0,\]
where $L_n$ is as in \eqref{cutoff}.
\end{lemma}

\begin{proof}
Let $\mathcal{T}$ be a tree of depth $L_n<d\left(\mathcal{T}\right)=l\leq L_{\text{max}}$. Then, for \[\Tilde{k}=\underset{(2k,l)\in \mathcal{T}}{\min} k,\qquad \epsilon=\epsilon\left(\Tilde{k},l-1\right),\]  let $\mathcal{T}^-$ be the corresponding tree whose nodes $(l, 2\Tilde{k})$ and $(l, 2\Tilde{k}+1)$ have been removed, i.e. $\mathcal{T}=\mathcal{T}^-\cup\left\{(l, 2\Tilde{k}),(l, 2\Tilde{k}+1)\right\}$. From  \eqref{def:nux} and \eqref{def:rox}, we have

\begin{align}
\label{T- to T}
\begin{split}
\Pi[\mathcal{T} \given X] &= \Pi[\mathcal{T}^- \given X] \frac{p_{\epsilon}^{X}}{1-p_{\epsilon}^{X}}\left(1-p^X_{\epsilon_0}\right)\left(1-p^X_{\epsilon1}\right)\\
&= \Pi[\mathcal{T}^- \given X]    p_{\epsilon}    \frac{\left(1-p_{\epsilon_0}\right)\left(1-p_{\epsilon1}\right)}{1-p_{\epsilon}}  \nu_\veps^X\\
&\leq \left(1-\Gamma^{-L_n}\right)^{-1} \Pi[\mathcal{T}^- \given X]  \frac{ 2^{N_X(I_{\veps})}}{\Gamma^{l+1}}  
\underbrace{  \frac{B(a+N_X(I_{\veps0}),a+N_X(I_{\veps1}))}{
B(a,a)} }_{\eqqcolon Q}.
\end{split}
\end{align}
Then, from Lemma \ref{bound_gamma}, we have for $\Tilde{n}_0=N_{X}\left(I_{\epsilon0}\right)$, $\Tilde{n}_1=N_{X}\left(I_{\epsilon1}\right)$ and $\Tilde{n}=N_{X}\left(I_{\epsilon}\right)$, that

\[
Q \lesssim
\underbrace{ 
\frac{\left(2a+\Tilde{n}_1-1/2\right)^{\Tilde{n}_1}	    \left(2a+\Tilde{n}_2-1/2\right)^{\Tilde{n}_2}}{ \left(2a+\Tilde{n}-1/2\right)^{\Tilde{n}} } 
}_{\eqqcolon Q_1}
\underbrace{ 
\frac{\left(2a+\Tilde{n}_1-1/2\right)^{a-1/2}	    \left(2a+\Tilde{n}_2-1/2\right)^{a-1/2}}{ \left(2a+\Tilde{n}-1/2\right)^{2a-1/2} } 
}_{\eqqcolon Q_2}.
\]
Under our assumptions on $f_0$, on the event $\mathcal{B}_M$ and for $n$ large enough, 
\[n_{X}\left(I_{l,k}\right) \geq \frac{\mu}{2}n2^{-l} \to\infty\]
for any $l\leq L_{\text{max}}$. Under the same conditions,
\[ \left|\Tilde{n}_1-\Tilde{n}_2\right| \leq n\left| P_0(I_{\veps0}) - P_0(I_{\veps1}) \right| + 2M M_{n,l}\leq nK2^{-l(1+\alpha)} +  2MM_{n,l}.\]
The last inequality stems from the fact that $f_0$ is $\alpha$-Hölder regular. Therefore, on $\mathcal{B}_M$, for $n$ large enough, if we note $v_{\Tilde{n}_1, \Tilde{n}_2}= \Tilde{n}_1-\Tilde{n}_2$, since $\Tilde{n}=\Tilde{n}_1+\Tilde{n}_2$ and $\log(1+x)\leq x$ for $x>-1$,

\begin{align*}
Q_1 &= \exp \Bigg( \Tilde{n}_1 \log\left(\frac{1}{2}+\frac{\Tilde{n}_1-\Tilde{n}_2+2a-1/2}{2(2a-1/2+\Tilde{n})}\right) + \Tilde{n}_2 \log\left(\frac{1}{2}-\frac{\Tilde{n}_1-\Tilde{n}_2-2a+1/2}{2(2a-1/2+\Tilde{n})}\right) \Bigg)  \\
& = \frac{1}{2^{\Tilde{n}}}  \exp \left( \Tilde{n}_1 \log\Big( 1+\frac{v_{\Tilde{n}_1, \Tilde{n}_2}+2a-1/2}{2a-1/2+\Tilde{n}} \Big) + \Tilde{n}_2 \log\Big(1-\frac{v_{\Tilde{n}_1, \Tilde{n}_2}-2a+1/2}{2a-1/2+\Tilde{n}}\Big) \right)\\
&\leq \frac{1}{2^{\Tilde{n}}}  \exp \left( \frac{v_{\Tilde{n}_1, \Tilde{n}_2}^2}{2a-1/2+\Tilde{n}}  +\frac{\Tilde{n}(2a-1/2)}{2a-1/2+\Tilde{n}} \right) \le  \frac{C}{2^{\Tilde{n}}}  \exp \left(  \frac{8K^2n^2 2^{-2l(1+\alpha)}}{\mu n2^{-l}}   + \frac{16M^2M_{n,l}^2}{\mu n2^{-l}} \right) \\
&\leq  \frac{C}{2^{\Tilde{n}}}  \exp \left( \Big(8K^2\mu^{-1}c_0^{-1-2\alpha}  + 32M^2(\mu \log 2)^{-1}\Big)\log n\right).
\end{align*}
The last inequality stems from $l>L_n$ and the definition of $L_n$. The last factor is even easier to control as, on $\mathcal{B}_M$,

\[
Q_2 \lesssim \left[n2^{-l}\right]^{-1/2} \lesssim n^{-\frac{\alpha}{1+2\alpha}} \log(n)^{-\frac{1/2}{1+2\alpha}}.
\]
Finally, this leads us to
\[
\Pi[\mathcal{T} \given X] =o\left( \Pi[\mathcal{T}^- \given X] \frac{n^{\left(8K^2\mu^{-1}c_0^{-1-2\alpha}  + 32M^2(\mu \log 2)^{-1}-\frac{\alpha}{1+2\alpha}\right)}}{\Gamma^{l}}\right)
\]
uniformly on $\mathcal{T}$ such that $L_n<d\left(\mathcal{T}\right)=l\leq L_{\text{max}}$.
The application $\mathcal{T}\longrightarrow \mathcal{T^-}$ defined above is surjective and  is such that each tree $\mathcal{T^-}$ is the image of at most $2^{l-1}$ trees $\mathcal{T}$. Then, the event of interest verifies for $\Gamma>2$ and $\bar{C}=8K^2\mu^{-1}c_0^{-1-2\alpha}  + 32M^2(\mu \log 2)^{-1}-\frac{\alpha}{1+2\alpha}$,

\begin{align}
\label{bnd_depth_prob}
\begin{split}
& \Pi[d(\mathcal{T})  > L_n \given X] = \sum_{l=L_n+1}^{L_\text{max}} \Pi[d(\mathcal{T}) = l \given X] =  \sum_{l=L_n+1}^{L_\text{max}}  \sum_{\mathcal{T}: d(\mathcal{T}) = l} \Pi[ \mathcal{T} \given X] \\
&= o\left( \sum_{l=L_n+1}^{L_\text{max}}  \sum_{\mathcal{T}: d(\mathcal{T}) = l} \Pi[\mathcal{T}^- \given X] \frac{n^{\bar{C}}}{\Gamma^{l}} \right) 
 =o\left( \sum_{l=L_n+1}^{L_\text{max}}  \sum_{\mathcal{T^-}} \Pi[\mathcal{T}^- \given X] \frac{2^l n^{\bar{C}}}{\Gamma^{l}} \right)
=o\left( \frac{2^{L_n} n^{\bar{C}}}{\Gamma^{L_n}}\right)
\end{split}
\end{align}
which is $o(1)$ whenever
$\{\log \Gamma/\log 2 -1\}/(1+2\alpha)\ge \bar{C},$ 
that is, if $\Gamma\geq 2^{1+\bar{C}(1+2\alpha)}$.
\end{proof}

\begin{lemma}
\label{huge signals taken}
Under the same assumptions on $f_0$ as in Lemma \ref{not too deep}, for $\Pi$ as in Theorem \ref{contraction_rate} and on the events $\mathcal{B}_M$ from Lemma \ref{obs_event}, for $\tau>0$ large enough and $L_n$ as in \eqref{cutoff}, the set
\[S(f_0, \tau)\coloneqq \left\{ (l,k): |f_{0,lk}|  \geq \tau\frac{\log n}{\sqrt{n}} \right\}\]
satisfies, as $n\to\infty$,
\[\Pi[ \{\mathcal{T}: S(f_0, \tau) \not\subset \mathcal{T}_{\text{int}}\} \given X] \to 0.\]
\end{lemma}

\begin{proof}
First, since $f_0\in\Sigma(\alpha,K)$ for some $\alpha,K>0$, there exists $C>0$ such that, for any $l\geq 0, 0\leq k<2^l$, $ |f_{0,lk}|\leq C2^{-l(\alpha+1/2)}$. Thus, for $\tau$ large enough, $(l,k)\in S(f_0, \tau)$ implies $l\leq L_n$.\\
Now, let's take $(l_S,k_S)$ a node in $S(f_0, \tau)$. Then, let's define \[\mathbb{T}_{n,(l_S,k_S)}\coloneqq\{\mathcal{T}\in\mathbb{T}_n\ |\ (l_S,k_S) \notin \mathcal{T}_{\text{int}} \},\] the set of trees in the support of our prior distribution on tree structures that do not have $(l_S,k_S)$ as an internal node, and $\epsilon=\epsilon\left(k_S,l_S\right)$. To any tree $\mathcal{T}\in\mathbb{T}_{n,\left(l_S,k_S\right)}$, it is possible to associate the full binary tree $\mathcal{T}^+$ which is the smallest extension of $\mathcal{T}$ with $(l_S,k_S)$ as an interior node,
\[ \mathcal{T}^+=\underset{\mathcal{T}'\in\mathbb{T}_n:\ \mathcal{T}\subset \mathcal{T}',\  (l_S,k_S)\in \mathcal{T}_{\text{int}}' }{\argmin} \left|\mathcal{T}'\right|.\]
 This new tree is realized with the completion of the route from the root to the node $(l_S,k_S)$, starting from the leaf node $(l_0,k_0)$ of this route which is included in  $\mathcal{T}$. Then, as in \eqref{T- to T} and using Lemma \ref{bound_gamma}, we now have for some constant $C>1$,

\begin{align}
\label{T to T+}
\begin{split}
\frac{\Pi[\mathcal{T} \given X]}{\Pi[\mathcal{T}^+ \given X]} &\leq C^{{l_S}^2}  \prod_{l=l_0}^{l_S}  \Bigg(2^{n_{X}\left(I_{\epsilon^{[l]}}\right)} \mathrm{B}\Big(a+ n_{X}\left(I_{\epsilon^{[l]}0}\right), a+ n_{X}\left(I_{\epsilon^{[l]}1}\right)\Big) \Bigg)^{-1}\\
&\leq C^{{l_S}^2}  \medmath{ \underbrace{   \prod_{l=l_0}^{l_S}   \frac{(2a + n_{X}\left(I_{\epsilon^{[l]}}\right) - 1/2)^{2a-1/2}}{(a + n_{X}\left(I_{\epsilon^{[l]}0}\right) - 1/2)^{a-1/2} (a + n_{X}\left(I_{\epsilon^{[l]}1}\right) - 1/2)^{a-1/2}}}_\text{$\eqqcolon Q_1$}}\\
&\qquad  \medmath{ \underbrace{   \prod_{l=l_0}^{l_S}   \frac{(2a + n_{X}\left(I_{\epsilon^{[l]}}\right) - 1/2)^{n_{X}\left(I_{\epsilon^{[l]}}\right)}}{2^{n_{X}\left(I_{\epsilon^{[l]}}\right)}(a + n_{X}\left(I_{\epsilon^{[l]}0}\right) - 1/2)^{n_{X}\left(I_{\epsilon^{[l]}0}\right)} (a + n_{X}\left(I_{\epsilon^{[l]}1}\right) - 1/2)^{n_{X}\left(I_{\epsilon^{[l]}1}\right)}}}_\text{$\eqqcolon Q_2$}}.
\end{split}
\end{align} where we recall that $\epsilon^{[l]}$ denotes the $l$ first elements of the sequence $\epsilon$.
On the event $\mathcal{B}_M$, for all $l\leq L_n+1$ and possible $k$, we have, using that $f_0\geq \mu>0$, $N_{X}\left(I_{l,k}\right) \gtrsim n2^{-l} \gtrsim n2^{-L_n} \to \infty$ as $n\to\infty$. Since it is also upper bounded (as $f_0$ is a Hölder density), we have $N_{X}\left(I_{l,k}\right)\lesssim n2^{-l}$. Therefore, since these bounds are uniform on $l\leq L_n+1$,

\[
Q_1 \leq  \prod_{l=l_0}^{l_S}  C \left(n2^{-l}\right)^{1/2}\leq C^{{l_S}}\sqrt{n}^{l_S}.
\]
Also, in $Q_2$, the factor at index $l$ is equal to, writing $\tilde{n}_0=N_{X}\left(I_{\epsilon^{[l]}0}\right), \tilde{n}_1=N_{X}\left(I_{\epsilon^{[l]}1}\right), \tilde{n}=N_{X}\left(I_{\epsilon^{[l]}}\right)$,
\[\exp\left[\tilde{n}_0\log\Bigg(\frac{2a-1/2+\tilde{n}}{2a-1+2\tilde{n}_0}\Bigg)+\tilde{n}_1\log\Bigg(\frac{2a-1/2+\tilde{n}}{2a-1+2\tilde{n}_1}\Bigg)\right].\]
If we write $KL(a;b)$ the Kullback-Leibler divergence between Bernoulli distributions of parameters $0\leq a,b\leq1$, then, for $n$ large enough, on $\mathcal{B}_M$, this is bounded by 
\[\exp\left[-C\tilde{n}KL\Big(\frac{a-1/2+\tilde{n}_0}{2a-1+\tilde{n}};1/2\Big)\right]\exp\left[\tilde{n}\log\Big(1+\frac{1}{4a-2+2\tilde{n}}\Big)\right].\]
The second factor can be bounded by a constant, uniformly on $l\leq L_n+1$. The first factor can be bounded by $1$ for $l<l_S$, while for $l=l_S$, we can use the bound $KL(a;b)\geq \norm{\text{Be}(a)-\text{Be}(b)}^2_1/2$ to write
\[ \exp\left[-C\tilde{n}KL\Big(\frac{a-1/2+\tilde{n}_0}{2a-1+\tilde{n}};1/2\Big)\right]\leq \exp\left[-C\tilde{n}^{-1}(\tilde{n}_0-\tilde{n}_1)^2\right].\]
By definition $\left|f_{0,l_Sk_S}\right|=2^{l_S/2}\left|P_0(I_{(l_S+1)(2k+1)})-P_0(I_{(l_S+1)(2k)})\right|$, so that on $\mathcal{B}_M$, $\left|\tilde{n}_0-\tilde{n}_1\right|\geq n\left|f_{0,l_Sk_S}\right|2^{-l_S/2}-2MM_{n,l_S+1}$, hence the upper bound for $\tau$ large enough:\[\exp\left[-C(\tau\log n -2M\sqrt{l_S+1+L_n})^2\right]\leq \exp\left[-C\tau^2\log^2 n\right] ,\]
where we used the definition of $S$, $M_{n,l_S+1}$, $L_n$ and $l_S\leq L_n$. 

Finally, for $\tau$ large enough and using that $l_S\leq L_n \leq \log n$, we can conclude that there exists constants $C_1,C_2>0$ such that

\begin{equation}
\frac{\Pi[\mathcal{T} \given X]}{\Pi[\mathcal{T}^+ \given X]} \leq C_1^{l_S^2}n^{-(C_2\tau^2-1/2)\log n} \leq n^{-(C_2\tau^2-1/2-\log C_1)\log n}.
\end{equation}
Since any tree verifying  $(l_S,k_S) \in \mathcal{T}$ is the image of at most $l_S+1$ trees by the map 
\begin{align*}
\mathbb{T}_{n,(l_S,k_S)}&\rightarrow \left\{\mathcal{T}'\in \mathbb{T}_n:\ (l_S,k_S) \in \mathcal{T}_{\text{int}}'\right\}&\\
\mathcal{T}&\mapsto \mathcal{T}^+&,
\end{align*}
as it is the length of the path from the root to the node $(l_S,k_S)$ in a tree $\mathcal{T}\in\mathbb{T}_n$,
\begin{align*}
\begin{split}
\Pi[ (l_S,k_S) \notin \mathcal{T} \given X]  &=  \sum_{ \mathcal{T}: (l_S,k_S) \notin \mathcal{T}} \frac{\Pi[\mathcal{T} \given X] }{\Pi[\mathcal{T}^+ \given X] } \Pi[\mathcal{T}^+ \given X] \\
& \leq n^{-(C_2\tau^2-1/2-\log C_1)\log n}  (l_S+1)  \sum_{ \mathcal{T}: (l_S,k_S) \in \mathcal{T}}   \Pi[\mathcal{T} \given X] \\
&\leq n^{-(C_2\tau^2-1/2-\log C_1)\log n}  \log n,
\end{split}
\end{align*}
which allows us in conjunction with the definition of $L_n$ to conclude that

\begin{align*}
\Pi[ \{\mathcal{T}: S(f_0, \tau) \not\subset \mathcal{T}\} \given X] &\leq \sum_{(l,k)\in S(f_0, \tau)}  \Pi[ (l,k) \notin \mathcal{T} \given X] \\
&\leq  2^{L_n+1 }n^{-(C_2\tau^2-1/2-\log C_1)\log n}  \log n\\
&\to 0
\end{align*}
as $n\to\infty$ for $\tau$ large enough.
\end{proof}

\begin{lemma}
\label{no distortion}

Let $T_n = \{\mathcal{T}\in\mathbb{T}_n :\ d\left(\mathcal{T}\right) \leq L_n,\  S(f_0, \tau) \subset \mathcal{T}   \}$ for $L_n$ as in \eqref{cutoff}, $c_0>0$ small enough, and $\tau>0$ as in Lemma \ref{huge signals taken}. Then, under the conditions of Lemma \ref{not too deep} and on the event $\mathcal{B}_M$ for $M>0$ large enough, there exists a constant $C>0$ such that for $n$ sufficiently large, uniformly on $\mathcal{T}\in T_n$,

\[\int \underset{(l,k)\in \mathcal{T}_\text{int}}{\max} \left| f_{lk} - f_{0,lk}\right|d\Pi[ f \given \mathcal{T}, X] \leq C\sqrt{\frac{\log n}{n}}.\]

\end{lemma}

\begin{proof}
Given a tree $\mathcal{T}$, let us define the map $\Bar{f}_{\mathcal{T}}$ such that, for each terminal node $(l,k)$ in $\mathcal{T}_{\text{ext}}$ and $x\in I_{lk}$,\[\Bar{f}_{\mathcal{T}}(x)=2^l\prod_{i=1}^{l}\Bar{Y}_{\epsilon^{[i]}}, \quad \epsilon=\epsilon\left(k,l\right),\]where \[\Bar{Y}_\epsilon= E\left[Y_\epsilon\given X, \mathcal{T}\right]=\frac{a+N_{X}(I_{\epsilon0})}{2a+N_X(I_\epsilon)}.\]This defines the mean posterior density given the tree structure $\mathcal{T}$. Similarly, for each $(l,k)\in\mathcal{T}$, with $\epsilon=\epsilon\left(k,l\right)$,  the mean probability measure of $I_{\epsilon}$ is \[\Bar{P}(I_{\epsilon})=\prod_{i=1}^{|\epsilon|}\Bar{Y}_{\epsilon^{[i]}}\eqqcolon \Bar{p}_\epsilon.\]

Then, expressing the coefficients of the decomposition in the Haar wavelet basis of this mean posterior density, we obtain that for each $(l,k)\in \mathcal{T}_\text{int}$, $\epsilon=\epsilon\left(k,l\right)$,
\[ \Bar{f}_{\mathcal{T},lk} \coloneqq \langle \Bar{f}_{\mathcal{T}}, \psi_{lk}\rangle = 2^{l/2} \left(\Bar{p}_{\epsilon}-2\Bar{p}_{\epsilon0}\right) = 2^{l/2}\Bar{p}_{\epsilon}\left(1-2\Bar{Y}_{\epsilon0}\right),\] while $\Bar{f}_{\mathcal{T},lk} =0$ for $(l,k)\not\in \mathcal{T}_\text{int}$.
When it comes to the true sampling density $f_0$, we obtain the similar expression, denoting $p_{0,\epsilon}\coloneqq P_0(I_\epsilon)$ and $y_{\epsilon0}\coloneqq\frac{P_0\left(I_{\epsilon0}\right)}{P_0\left(I_{\epsilon}\right)}$,
\[ f_{0, lk} = 2^{l/2}p_{0,\epsilon}(1-2y_{\epsilon0}),\]
and, for densities $f$ sampled from the posterior distribution given $\mathcal{T}$, with $p_\epsilon\coloneqq \prod_{i=1}^{|\epsilon|}Y_{\epsilon^{[i]}}$,
\[ f_{lk} = 2^{l/2}\Tilde{p}_\epsilon\left(1-2Y_{\epsilon0}\right)\mathds{1}_{(l,k)\in\mathcal{T}_{\text{int}}}.\]
From now on, for simplicity of notations, $\epsilon=\epsilon(k,l)$ as the context will make it clear what the pair $(l,k)$ is. For any $\mathcal{T}\in \mathbb{T}_n$, one can bound $\left|f_{lk} - f_{0,lk}\right|\leq \left|f_{lk} - \Bar{f}_{\mathcal{T},lk}\right| + \left|\Bar{f}_{\mathcal{T},lk} - f_{0,lk}\right|$.
Using the above expressions, the second term is rewritten as
\[ \left|\Bar{f}_{\mathcal{T},lk} - f_{0,lk}\right| = \Bigg|  f_{0,lk}\left[ \frac{\Bar{p}_\epsilon}{p_{0,\epsilon}}  -1\right]  + 2^{l/2+1}(y_{\epsilon0}-\Bar{Y}_{\epsilon0})\Bigg|.  \]
Then, as we are on the event $\mathcal{B}_M$, we bound the two terms above by means of Lemmas 1 and 2 from \cite{c17} (which are valid for some $c_0$ small enough) and the bound $p_{0,\epsilon}\lesssim 2^{-|\epsilon|}$ (as $f_0$ is upper bounded), which give uniformly on $\mathcal{T}\in \mathbb{T}_n$ and $(l,k)\in \mathcal{T}_{\text{int}}$,

\begin{align}
\label{bound bias}
\begin{split}
\left|\Bar{f}_{\mathcal{T},lk} - f_{0,lk}\right| &\lesssim |f_{0,lk}| \left[    a\frac{2^l}{n}+ \sqrt{\frac{L_n2^l}{n}} \right]  + \left[ |f_{0,lk}| \frac{a2^l}{n} + \sqrt{\frac{L_n}{n}} \right]\\
&\lesssim  |f_{0,lk}| \left[    a\frac{2^l}{n}+ \sqrt{\frac{L_n2^l}{n}} \right]  + \sqrt{\frac{\log n}{n}} \qquad \text{as $L_n\lesssim \log n$}.
\end{split}
\end{align}
Since $f_0$ is $\alpha$-Hölder, $|f_{0,lk}|\lesssim 2^{-l(1/2+\alpha)}$, and the last quantity in the above inequality is smaller (up to a constant) than $\sqrt{n^{-1}\log n}$ as $l\leq L_n$. It then remains to bound the term

\[\int \underset{(l,k)\in \mathcal{T_\text{int}}}{\max} \left|f_{lk} - \Bar{f}_{\mathcal{T},lk} \right|d\Pi[ f \given \mathcal{T}, X]. \]
To do so, let's first define the event
\[\mathcal{A}= \underset{\epsilon: |\epsilon| < L_n}{\cap} \left\{  |\Bar{Y}_{\epsilon0} - Y_{\epsilon0} |\leq M'\sqrt{\frac{L_n}{nP_0(I_{\epsilon0})}}  \right\}\]for $M'>0$. By Lemma \ref{beta control}, it follows that, for $d$ a small constant, 
\begin{align}
\label{quick decrease event}
\begin{split}
\Pi\left[\mathcal{A}^c \given \mathcal{T}, X\right]&\lesssim \sum_{l\leq L_n} 2^l \exp(-C{M'}^2\log n)
\lesssim  2^{L_n} \exp(-C{M'}^2\log n),
\end{split}
\end{align}
which is smaller than $\left(n/\log n\right)^{1/(1+2\alpha)}n^{-C{M'}^2}$. Then,

\[\left|f_{lk} - \Bar{f}_{\mathcal{T},lk} \right| = \Bigg| 2^{l/2+1}\Bar{p}_{\epsilon}\left(\Bar{Y}_{\epsilon0}-Y_{\epsilon0}\right) +\left[\frac{p_{\epsilon}}{\Bar{p}_{\epsilon}}-1\right]\left(\Bar{f}_{\mathcal{T},lk} + 2^{l/2+1}\Bar{p}_{\epsilon}(\Bar{Y}_{\epsilon0}-Y_{\epsilon0})\right) \Bigg|. \]
Applying Lemmas 2 and 3 from \cite{c17} (valid once again for some $c_0$ small enough), on the events $\mathcal{B}_M$ and $\mathcal{A}$, uniformly on $\epsilon$ such that $|\epsilon|=l$ for some $l\leq L_n$,
\[  \left|\frac{p_{\epsilon}}{\Bar{p}_{\epsilon}}-1\right| \lesssim \sum_{i=1}^{l} \sqrt{\frac{L_n}{nP_0(I_{\epsilon^{[i]}})}}  \lesssim \sqrt{  \frac{L_n2^l}{n}  } . \]
Therefore, we directly have that on the events $\mathcal{B}_M$ and $\mathcal{A}$, 

\begin{align}
\begin{split}
\left|f_{lk} - \Bar{f}_{\mathcal{T},lk} \right| &\lesssim \left| \Bar{f}_{\mathcal{T},lk} \right| \sqrt{  \frac{L_n2^l}{n}  } + 2^{l/2} \Bar{p}_{\epsilon} \left[ \sqrt{  \frac{L_n}{nP_0(I_{\epsilon0})} }+ \frac{L_n}{n}\sqrt{ \frac{2^l}{P_0(I_{\epsilon0})}    }  \right] \\
& \lesssim  \left| \Bar{f}_{\mathcal{T},lk} \right| \sqrt{  \frac{L_n2^l}{n}  } + \sqrt{\frac{L_n}{n}},
\end{split}
\end{align}
where we used that on $\mathcal{B}_M$, $\Bar{p}_{\epsilon}\lesssim 2^{-|\epsilon|}$ for $n$ large enough as $f_0$ is upper bounded, and $P_0(I_{\epsilon0})\gtrsim 2^{-|\epsilon|}$. Finally, with  $\left| \Bar{f}_{\mathcal{T},lk} \right|\leq \left|\Bar{f}_{\mathcal{T},lk} - f_{0,lk}\right| + \left|f_{0,lk}\right|$ and using the same computation as for \eqref{bound bias}, we have $\left|f_{lk} - \Bar{f}_{\mathcal{T},lk} \right| \lesssim \sqrt{ \frac{\log n}{n}  }$. This gives

\begin{align}
\begin{split}
\int& \underset{(l,k)\in \mathcal{T_\text{int}}}{\max}  \left|f_{lk} - f_{0,lk}\right|d\Pi\left[ f \given \mathcal{T}, X\right]  \lesssim \sqrt{\frac{\log n}{n}} + \int_{\mathcal{A}^c} \underset{(l,k)\in \mathcal{T_\text{int}}}{\max} \left|f_{ lk} - \Bar{f}_{\mathcal{T},lk} \right|d\Pi[ f \given \mathcal{T}, X]  \\
&\lesssim \sqrt{\frac{\log n}{n}} + 2^{L_n/2}\Pi[\mathcal{A}^c \given \mathcal{T}, X]
\lesssim \sqrt{\frac{\log n}{n}} + \left(\frac{n}{\log n}\right)^{\frac{\alpha/2}{2\alpha+1}} \left(\frac{n}{\log n}\right)^{\frac{1}{1+2\alpha}}n^{-dM'^2}\\
&\lesssim \sqrt{\frac{\log n}{n}} \qquad \text{for $M'$ large enough,}
\end{split}
\end{align}
where the second inequality comes from the fact that, for a density $f$, $\left| \langle f, \psi_{lk}\rangle\right| \leq 2^{l/2}$. This concludes the proof as this bound holds uniformly on $\mathcal{T}\in T_n$.
\end{proof}

\subsection{Proofs for confidence bands}

\begin{proof}[Proof of Proposition \ref{confidence_set_level}]
On the event $\mathcal{E}$ from Lemma \ref{median_tree_prop}, the bound on the median tree depth implies that for any $h,g\in C_n$,
\begin{align*}
\norm{h-g}_\infty&\leq \norm{h-f_{\mathcal{T}^*}}_\infty + \norm{g-f_{\mathcal{T}^*}}_\infty\\
&\leq 2  \sigma_n\\
&\leq 2A^{1/2} v_n\sqrt{\frac{\log n}{n}} 2^{L_n/2} \lesssim v_n \left(\frac{\log n}{n}\right)^{\frac{\alpha}{2\alpha+1}}.
\end{align*}
Also, Lemma \ref{lemma: supnorm convergence median tree} ensures that 
\[ \| \hat f_{\cT^*}-f_0 \|_\infty =O_{P_0}\left(\left(\frac{\log^2{n}}{n}\right)^{\frac{\al}{2\al+1}}\right).\]
Then, according to the proof of Proposition 3 in \cite{hoffmann2011}, for any $f_0\in\cS(\alpha, K, \eta)$ and $l_1$ large enough
\[ \underset{(l,k):\ l\geq l_1}{\sup} |\langle f_{0}, \psi_{lk}\rangle| \geq C2^{-l_1(\alpha+1/2)}.\]
For $\Delta_n>0$ and $\zeta>0$ such that
\[ \zeta\left(\frac{n}{\log^2 n}\right)^{1/(2\alpha+1)}\leq 2^{\Delta_n}\leq 2\zeta\left(\frac{n}{\log^2 n}\right)^{1/(2\alpha+1)},\]
this implies that 
\[\underset{(l,k):\ l\geq \Delta_n}{\sup}|\langle f_{0}, \psi_{lk}\rangle| \geq C \zeta^{-\alpha-1/2}\frac{\log n}{\sqrt{n}}.\]
Therefore, if $\zeta$ is small enough, there exists $l\geq\Delta_n$ and $0\leq k<2^l$ such that $|\langle f_{0}, \psi_{lk}\rangle|>A\log n/\sqrt{n}$, and then $(l,k)\in\mathcal{T}^*$ on $\mathcal{E}$ according to Lemma \ref{median_tree_prop}. As a consequence,
\begin{equation}\label{eqn: lower bound sigma_n} \sigma_n\geq v_n\sqrt{\frac{\log n}{n}} 2^{\Delta_n/2}\geq C' \frac{v_n}{\log^{1/2}n}\left(\frac{\log^2n}{n}\right)^{\alpha/(2\alpha+1)},\end{equation}
and since $\log^{1/2}n=o(v_n)$,  $\norm{f_0-f_{\mathcal{T}^*}}_\infty\leq \sigma_n/2$ for $n$ large enough. This allows us to conclude that 
\[P_0\left[f_0\in \cC_n\right]=P_0\left[\{f_0\in \cC_n\}\cap \mathcal{E}\right] +o(1)= 1+o(1).\]

It remains to determine the credibility level of the set $\cC_n$. From Theorem \ref{contraction_rate} and Lemma \ref{lemma: supnorm convergence median tree}, the posterior contracts towards $f_0$ and the $\hat f_{\cT^*}$ converges to $f_0$ on an asymptotically certain event $\mathcal{E}$, both at a faster rate than $\sigma_n$ (see \eqref{eqn: lower bound sigma_n}). Therefore, an application of the triangular inequality gives
\[\Pi\left[ \cC_n \given X \right] \geq \Pi\left[ \norm{f-f_0}_\infty \leq \sigma_n/2 \given X\right] \mathds{1}_{\mathcal{E}} + \Pi\left[ \cC_n \given X\right] \mathds{1}_{\mathcal{E}^c} = 1+o_{P_0}(1).\]
\end{proof}

\begin{proof}[Proof of Proposition \ref{prop:cmulti}]
The credibility statement follows from the fact that $\cC_n$ (respectively the multiscale ball) has credibility $1$ (respectively $1-\gamma$) asymptotically.  The diameter statement follows from the inclusion $\cC_n^{\cM}\subset\cC_n$. For coverage, one combines Theorem \ref{confidence_set_level} which gives that $\cC_n$ has asymptotic coverage $1$, with Theorem 5 in \cite{cn14} which from the nonparametric BvM (Theorem \ref{BVM}) enables to deduce frequentist coverage of $\|\cdot\|_{\cM_0(w)}$--balls (hence the multiscale ball in the intersection defining $\cC_n^{\cM}$ has asymptotic coverage $1-\ga$).
\end{proof}


\section*{Acknowledgments}
The authors would like to thank Li Ma for insightful comments.
\bibliographystyle{siam}
\bibliography{biblio}

\pagebreak

\appendix
\begin{center}\textcolor{header1}{\textbf{\huge Supplementary material}}\end{center}
\vspace{0.3in}

\section{The classical P\'olya tree and $T$--P\'olya trees}\label{section: polya tree variants}

Let us partition the sample space $I_\varnothing=[0,1)$  as $I_{1,0}\cup I_{1,1}$, these two subsets being the level-$1$ elementary regions. These can in turn be partitioned as $I_{1,0}=I_{2,0}\cup I_{2,1}$ and $I_{1,1}=I_{2,2}\cup I_{2,3}$, involving level-$2$ elementary regions. Continuing this partitioning scheme gives the general level-$k$ elementary region, $k\geq1$, whose set will be written as $\mathcal{A}^k$. More precisely, we partition $I_{l,k}=I_{l+1,2k}\cup I_{l+1, 2k+1}$, $l\geq0, 0\leq k\leq 2^l-1$. From this recursive partitioning scheme, one defines a random recursive partition of $I_\varnothing$ and an associated random density.
 
 The Pólya Tree prior corresponding to the partitioning $\cup_{l=1}^\infty\mathcal{A}^l$ is the distribution on probability measure on $[0;1)$, whose samples are defined by the conditional probabilities
\begin{equation}\label{eq: Polya_mass_splitting}\epsilon\in\mathcal{E}^*,\ P\left(I_{\epsilon0}|I_\epsilon\right)= V_{\epsilon0}\sim \text{Beta}(\nu_{\epsilon0},\nu_{\epsilon0}) .\end{equation}
For an appropriate choice of Beta parameters $\nu_\epsilon,\ \epsilon\in\mathcal{E}^*$, samples from this prior actually extends almost surely to an absolutely continuous measure, so that it can be seen as a prior on densities. The Beta random variables $V_{\epsilon0}$ then corresponds to the share of the mass on $I_\epsilon$ that is allocated to $I_{\epsilon0}$. This mass allocation scheme is illustrated on Figure \ref{fig: Polya_tree}: the random mass of each interval $I_\epsilon$ is the product of Beta variables on the edges of the path from the root to the corresponding node. As a consequence, the random mass on $I_{\epsilon}$, $\epsilon\in\mathcal{E}^*$, is equal to $\prod_{i=1}^{|\epsilon|}V_{\epsilon^{[i]}}$.

\begin{figure}[!h]
\footnotesize
\centering
\begin{tikzpicture}[
very thick,
level 1/.style={sibling distance=8cm},
level 2/.style={sibling distance=3.5cm},
level 3/.style={sibling distance=1cm},
every node/.style={circle,solid, draw=black,thin, minimum size = 0.5cm},
emph/.style={edge from parent/.style={dashed,black,thin,draw}},
norm/.style={edge from parent/.style={solid,black,thin,draw}}
]
	\node [rectangle] (r){$I_{\varnothing}=[0;1)$}
		child {
			node [rectangle] (a) {$I_0=[0;1/2)$}
			child {
				node [rectangle] {$I_{00}=[0;1/4)$}
				child[emph] {
					node [rectangle, dotted] {}
				}
				child[emph] {
					node [rectangle, dotted] {}
				}
				edge from parent node[left, draw=none]{$Y_{00}\sim \text{Beta}\left(\nu_{00}, \nu_{01}\right)$}
			}
			child {
				 node [rectangle] {$I_{01}=[1/4;1/2)$}
				child[emph] {
					node [rectangle, dotted] {}
				}
				child[emph] {
					node [rectangle, dotted] {}
				}
				edge from parent node[right, draw=none]{$Y_{01}=1-Y_{01}$}
			}
			edge from parent node[left, draw=none]{$Y_0\sim \text{Beta}\left(\nu_{0}, \nu_{1}\right)\quad$}
		}
		child {
			node [rectangle] {$I_{1}=[1/2;1)$}
			child {
				node [rectangle] {$I_{10}=[1/2;3/4)$}
				child[emph] {
					node [rectangle, dotted] {}
				 }
				 child[emph] {
					node [rectangle, dotted] {}
				 }
				edge from parent node[left, draw=none]{$Y_{10}\sim \text{Beta}\left(\nu_{10}, \nu_{11}\right)$}
			}
			child {
				node [rectangle] {$I_{11}=[3/4;1)$}
				child[emph] {
					node [rectangle, dotted] {}
				}
				child[emph] {
					node [rectangle, dotted] {}
				}
				edge from parent node[right, draw=none]{$Y_{11}=1-Y_{10}$}
				}
			edge from parent node[right, draw=none]{$\quad Y_1=1-Y_0$}
		 };
\end{tikzpicture}
\caption{Pólya Tree process on the dyadic recursive partitioning, with splits at midpoints.}
\label{fig: Polya_tree}
\end{figure}
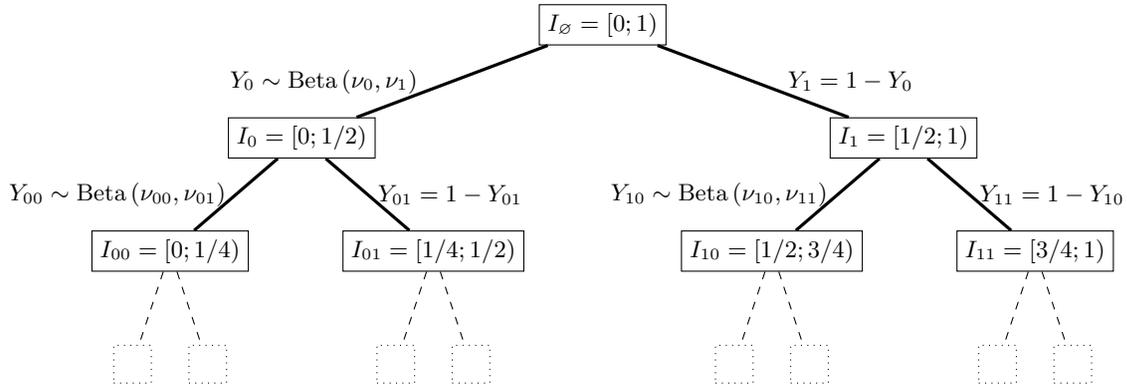

A simpler related prior on densities, the truncated Pólya Tree prior, stops the splitting of the mass at some level $L<\infty$ and has sampled densities which are constant on each set $I_{\epsilon}$ in $\mathcal{A}^L$, with value $\mu\left(I_{\epsilon}\right)^{-1}\prod_{i=1}^{|\epsilon|}V_{\epsilon^{[i]}}$. If one introduces the tree $T$ as \[T=\left\{(k,l),\ l\leq L, 0\leq k<2^l\right\},\]
that is the complete binary tree of depth $d(T)=L$, it corresponds to a T-Pólya tree distribution with $\Pi_\bT=\delta_{T}$.

\section{Tree posteriors: the Galton--Watson/P\'olya tree case} 
\label{sec: link_OPT_GW}
As shown in Subsection \ref{subsec: Partitioning}, the Markov process on trees $GW(p)$ can be seen as a distribution on partitions. We first show that it corresponds to the distribution introduced in  \cite{wm10}.

In the Optional Pólya Tree (OPT) construction, different recursive partitioning mechanism are allowed: each level-$k$ elementary region $A\in \mathcal{A}^k$ can be split in $M(A)$ different ways, the $j$-th being written as
\begin{equation}\label{eq: subdivision_elementary} A=\cup_{i=1}^{K_j(A)} A_{k}^j,\end{equation}where the $A_{k}^j$ are level-$(k+1)$ elementary regions (see Appendix \ref{section: polya tree variants}). Then, a random partition of the sample space $[0;1)$ is produced recursively. For $0\leq \rho([0,1))\leq 1$, the partition is the sample space itself with probability $\rho([0,1))$. Otherwise, one of the $M([0,1))$ partitions are drawn according to probability vector $\lambda([0;1))=\left(\lambda_1,\dots,\lambda_{M([0,1))}\right)$. The partitioning then continues: each elementary region $A$ stays intact with probability $\rho(A)$, otherwise it is split (a decision encoded by the variable $S(A)\sim \text{B}(\rho(A))$) and its partition is chosen according to probability vector $\lambda(A)$. Following the discussion in Subsection \ref{subsec: Partitioning}, the $GW(p)$ is a particular case, where $M(A)=1$, $\lambda(A)=1$ and $K_1(A)=2$, as the intervals are only ever split at their midpoints,
\begin{equation}\label{eq: subdivision_gwp} I_{l,k}=I_{l+1,2k}\cup I_{l+1,2k+1}.\end{equation}
The level-$k$ elementary regions are the $I_\epsilon$ with $|\epsilon|=k$.
 Also, it corresponds to the choice of \[\rho(I_{l,k})=1-p_{lk},\ l<L_{\text{max}},\qquad  \rho(I_{L_{\text{max}},k})=0.\]
 Given a partition $\mathcal{I}$, in OPT, a probability measure $Q$ is defined by the conditional probabilities, for $A$ an elementary region split as in \eqref{eq: subdivision_elementary},
 \[ \left(Q(A_1^j|A),\dots, Q(A_{K_j(A)}^j|A)\right)=Q(A)\theta(A),\qquad \theta(A)\sim \text{Dir}\left(\alpha_1^j(A),\dots,\alpha_{K_j(A)}^j(A)\right),\]
 with Dirichlet random variables $\theta$ mutually independent and independent from the variables $S(A')$ for $A \not\subset A'$, and $Q([0,1))=1$. For $M(A)=1$ and $K_1(A)=2$, it is similar to the mass allocation mechanism in \eqref{eq: Polya_mass_splitting} when $\alpha_1^1=\alpha_2^1=a$. However, whenever the recursive partitioning stops and gives a finite partition, these equations do not completely characterize a measure on Borelians of $[0,1)$, so that the measure $Q$ is defined on Borelians $B$ as
 \[ Q(B) = \sum_{A\in\mathcal{I}} Q(A) \frac{\mu\left(A\cap B\right)}{\mu(A)}.\]
This corresponds to the absolutely continuous measure with density constant on the elements of $\mathcal{I}$. Therefore, the distribution from Proposition \ref{opt} is actually a special case of OPT.

%

\section{The OPT posterior on trees} 
\label{sec: tree_GW_posterior}
In the following, we prove Propositions \ref{prop1} and \ref{prop2}. We first obtain a general formula for the posterior on trees, which implies an explicit formulation of $\Pi[\cdot\given X, \cT]$, and then focus on the OPT prior.
The posterior distribution on trees is given for $T\in\mathbb{T}_n$ by Bayes' formula as

\[ \Pi\left[T|X\right]=\frac{\int \Pi\left[X,T|f\right]d\Pi\left[f\right]}{\int \Pi\left[X|f\right]d\Pi\left[f\right]}.\]
Since $\Pi\left[X,T|f\right]=\mathds{1}_{\mathcal{T}=T}\prod_{i=1}^n f\left(X_i\right)$, the numerator is equal to

\[ \sum_{T'\in\mathbb{T}_n} \Pi\left[\mathcal{T}=T'\right]\mathds{1}_{\mathcal{T}=T} \int\prod_{i=1}^n f\left(X_i\right)d\Pi\left[f|T'\right]=\Pi\left[\mathcal{T}=T\right] \int\prod_{i=1}^n f\left(X_i\right)d\Pi\left[f|T\right].\]
Writing $N_T(X) \coloneqq \int\prod_{i=1}^n f\left(X_i\right)d\Pi\left[f|T\right]$ the marginal likelihood, the denominator can be expressed as

\[ \sum_{T'\in\mathbb{T}_n}  \Pi\left[\mathcal{T}=T'\right] \int \Pi\left[X, \mathcal{T}=T'|f\right]d\Pi\left[f\right] = \sum_{T'\in\mathbb{T}_n}  \Pi\left[\mathcal{T}=T'\right]  N_{T'}(X).\]
Let's compute $N_T(X)$. By definition, for any $i=1,\dots,n$,
\[ f\left(X_i\right) = \prod_{(l,k)\in T_{\text{ext}}} \left(\prod_{j=1}^l 2Y_{\epsilon(k,l)^{[j]}}\right)^{\mathds{1}_{X_i\in I_{lk}}},\]
and 
\begin{align*} 
\prod_{i=1}^n f\left(X_i\right) &= \prod_{(l,k)\in T_{\text{ext}}} \left(\prod_{j=1}^l 2Y_{\epsilon(k,l)^{[j]}}\right)^{N_X\left(I_{lk}\right)}\\
&= \prod_{(l,k)\in T\setminus\{(0,0)\}}  \left(2Y_{\epsilon(k,l)}\right)^{N_X\left(I_{lk}\right)}\\
&= \prod_{(l,k)\in T_{\text{int}}}  \left(2Y_{\epsilon(k,l)0}\right)^{N_X\left(I_{\epsilon(k,l)0}\right)}\left(2(1-Y_{\epsilon(k,l)0})\right)^{N_X\left(I_{\epsilon(k,l)1}\right)}.
\end{align*} 
On the one hand, we obtain that 
\begin{align*}
\Pi[f\given X, \cT]&= N_T(X)^{-1} \Pi[f, X \given \cT]=N_T(X)^{-1} \Pi[X \given f, \cT] \Pi[f \given \cT]\\
&=  C(X,T) \prod_{i=1}^n f\left(X_i\right) \prod_{(l,k)\in T_{\text{ext}}} \prod_{j=1}^l Y_{\epsilon(k,l)^{[j]}}^a \left(1-Y_{\epsilon(k,l)^{[j]}} \right)^a \\
&= C(X,T) \prod_{(l,k)\in T_{\text{int}}}  Y_{\epsilon(k,l)0}^{a+N_X\left(I_{\epsilon(k,l)0}\right)}\left((1-Y_{\epsilon(k,l)0})\right)^{a+N_X\left(I_{\epsilon(k,l)1}\right)},\\
\end{align*}for $C(X,T)$ a constant depending on $X$ and $T$ only, which proves the claim of Proposition \ref{prop1}.
On the other hand, for any variable $Y\sim \text{Beta}(a,a)$, one obtains
\[ E\left[Y^N(1-Y)^M\right]= \int_0^1 y^N(1-y)^M \frac{y^a(1-y)^a}{B(a.a)}dy = \frac{B(a+N,a+M)}{B(a,a)}.\]
Therefore, \[ N_T(X) =  \prod_{(l,k)\in T_{\text{int}}} 2^{N_X\left(I_{\epsilon(k,l)}\right)} \frac{B\left(a+N_X\big(I_{\epsilon(k,l)0}\big),a+N_X\big(I_{\epsilon(k,l)1}\big)\right)}{B(a,a)}.\]

Let's now focus on the special case of the $\operatorname{GW}(p)$ tree prior, as in Proposition \ref{prop2}. For any possible pair $(l,k)$, take $T\in\mathbb{T}_n$ such that $(l,k)\in T_{\text{ext}}$ and let \[T^+=T\cup\left\{(l+1,2k), (l+1,2k+1)\right\}.\] Then, 
\begin{equation}\label{eq: prior_T_Tplus} \Pi[T^+]=\Pi[T]\frac{p_{lk}}{1-p_{lk}}(1-p_{l+1,2k})(1-p_{l+1,2k+1}),\end{equation}
and 
\begin{align}
\begin{split}
\label{eq:posterior_proba_split}
\frac{\Pi[T^+|X]}{\Pi[T|X]} &= \frac{\Pi\left[\mathcal{T}=T^+\right] L_{T^+}(X)}{\Pi\left[\mathcal{T}=T\right] L_{T}(X)}\\
&= \frac{p_{lk}}{1-p_{lk}}(1-p_{l+1,2k})(1-p_{l+1,2k+1})\\
&\qquad2^{N_X\left(I_{\epsilon(k,l)}\right)}  \frac{B\left(a+N_X\big(I_{\epsilon(k,l)0}\big),a+N_X\big(I_{\epsilon(k,l)1}\big)\right)}{B(a,a)}.
\end{split}
\end{align}

This last quantity is independent of $T$ and $T^+$ and depends only on $(l,k)$. Therefore, if we can find $p_{lk}^X, p_{l+1,2k}^X, p_{l+1,2k+1}^X$ such that the last quantity in \eqref{eq:posterior_proba_split} is equal to
\[ \frac{p_{lk}^X}{1-p_{lk}^X}(1-p_{l+1,2k}^X)(1-p_{l+1,2k+1}^X),\] for any appropriate $(l,k)$, we obtain a formula similar to \eqref{eq: prior_T_Tplus} and the posterior on trees is a $GW(p^X)$ process. This defines a set of equations that has a solution, as for any $0\leq k< 2^{L_{\text{max}}}$, we necessarily have $p_{L_{\text{max}}k}=0$ and the equations can be solved to obtain $p^X$, starting from $l=L_{\text{max}}$ and solving the successive equations in a ``bottom--up" way up to the level $l=0$.

\section{Median tree properties}

\begin{lemma}
\label{median_tree_prop}
Under the same prior and assumptions as in Theorem \ref{contraction_rate}, there exists an event $\mathcal{E}$, such that $P_0\left[\mathcal{E}\right]=1+o(1)$, on which the following is true: for some constants $A>0,B>0$,
\begin{itemize}
\item $2^{d(\mathcal{T}^*)}\leq A2^{L_n}\asymp \left(n/\log n\right)^{1/(2\alpha+1)}$, $L_n$ as in \eqref{cutoff},
\item For any $(l,k)$ such that $|f_{0,lk}|\geq Bn^{-1/2}\log n$, $(l,k)\in \mathcal{T}_{\text{int}}^*$.
\end{itemize}
\end{lemma}
\begin{proof}
On the event $\mathcal{B}_M$ from Lemma \ref{obs_event}, Lemma  \ref{huge signals taken} shows that the set $\mathbb{T}^{(2)}$ of trees satisfying the second condition in the lemma, for $B$ large enough, is such that $\Pi\left[\mathbb{T}^{(2)}|\ X\right]\to1$. Therefore the event \[ \Tilde{\mathcal{E}} =\left\{\Pi\big[\mathbb{T}^{(2)}|\ X\big]\geq 3/4\right\}\supset \mathcal{B}_M\]
is asymptotically certain.\\
For any node $(l,k)$ such that $|f_{0,lk}|\geq Bn^{-1/2}\log n$, since it belongs to the interior nodes of any tree in $\mathbb{T}^{(2)}$  by definition,
\[ \Pi\left[(l,k)\in \mathcal{T}_{\text{int}}| \ X\right] = \sum_{\mathcal{T}\in\mathbb{T}_n:\ (l,k)\in \mathcal{T}_{\text{int}}} \Pi\left[\mathcal{T}|\ X\right]\geq \Pi\left[\mathbb{T}^{(2)}|\ X\right].\]
Then, on $\Tilde{\mathcal{E}}$, $(l,k)\in \mathcal{T}^*$ by definition and $\mathcal{T}^*$ satisfies the second condition of the lemma.\\
Let's now turn to the set $\mathbb{T}^{(1)}$ of trees satisfying the first condition in the lemma. Using the same arguments as for \eqref{bnd_depth_prob}, there exists $C>0$ such that for any $l$ such that $2^{l}\gtrsim 2^{L_n}$ and $\Gamma>0$ large enough, 
\[ \Pi\left[d\big(\mathcal{T}\big)>l\ |\ X\right] \leq n^C \left(2/\Gamma\right)^l,\]
which holds on the event $\mathcal{B}_M$. Then, since \[ \Pi\left[(l,k)\in \mathcal{T}_{\text{int}}| \ X\right] \leq  \Pi\left[d\left(\mathcal{T}\right)>l | \ X\right],\] Markov's inequality implies
\begin{align*}
P_0\left[\big\{\mathcal{T}^*\notin\mathbb{T}^{(1)}\big\}\cap \mathcal{B}_M\right]&= P_0\left[\big\{\exists (l,k):\ 2^{l}> A2^{L_n},\ (l,k)\in\mathcal{T}^*\big\}\cap \mathcal{B}_M\right]\\
&\leq \sum_{l:\ 2^{l}> A2^{L_n}}^{L_{\text{max}}} \sum_{0\leq k<2^l-1}  P_{0}\left[\big\{\Pi[(l-1,\lfloor k/2\rfloor)\in\mathcal{T}_{\text{int}}\ |\ X]>1/2\big\}\cap \mathcal{B}_M\right]\\
&\leq \sum_{l:\ 2^{l}> A2^{L_n}}^{L_{\text{max}}} 2 \sum_{0\leq k<2^l-1}   E_{0}\left[\Pi[(l-1,\lfloor k/2\rfloor)\in\mathcal{T}_{\text{int}}\ |\ X] \mathds{1}_{\mathcal{B}_M}\right]\\
&=o(1)\text{ for $\Gamma$ large enough}.
\end{align*}
One concludes by noting that $\mathcal{B}_M$ is asymptotically certain according to Lemma \ref{obs_event}, and $\mathcal{E}=\big\{\mathcal{T}^*\in\mathbb{T}^{(1)}\big\}\cap \mathcal{B}_M$ satisfies the conditions of the lemma.
\end{proof}

\begin{lemma}\label{lemma: med tree depth}
Let $0<\alpha \leq 1$, $K>0$, $\mu>0$ and $\eta>0$. Let $\Pi$ be the same prior as in Theorem \ref{confidence_set_level}, then for $f_0\in\cS(\alpha, K, \eta)\cap \mathcal{F}(\alpha,K,\mu)$, 
\[ \left(n/\log^2 n\right)^{1/(2\alpha+1)} \lesssim 2^{d(\cT^*)}\lesssim \left(n/\log n\right)^{1/(2\alpha+1)},\] on an event of probability converging to $1$.
\end{lemma}
\begin{proof}
Using the same argument as above \eqref{eqn: lower bound sigma_n}, we obtain the lower bound. Lemma \ref{median_tree_prop} gives the upper bound.
\end{proof}

\begin{lemma}
\label{lemma: coeff med tree}
 Let $f_0$ and $\ell_0$ be as in Theorem \ref{BVM}, $\Pi$ as in  Proposition \ref{median_tree_prop} and  $\hat f_{\cT^*}$ as defined in \eqref{medest}. The median tree estimator then satisfies 
 \[ 
\max_{l> \ell_0(n)} \max_k |\hat f_{\cT^*,lk}-f_{0,lk}| =O_{P_0}\left(\frac{\log{n}}{\sqrt{n}}\right).
\]
\end{lemma}
\begin{proof}
Let $Q=\max_{l> \ell_0(n)} \max_k |\hat f_{\cT^*,lk}-f_{0,lk}|$. On the event $\mathcal{E}$ from Proposition \ref{median_tree_prop}, one has for $B$ as in the proposition,
\[ Q\leq\left(B\frac{\log n}{n^{1/2}}\right) \vee \max_{(l,k)\in \cT_{\text{int}}^*, l> \ell_0(n)} |\hat f_{\cT^*,lk}-f_{0,lk}|.\] Indeed, for $(l,k)\not\in \cT_{\text{int}}^*$, we necessarily have $\hat f_{\cT^*,lk}=0$ and $|f_{0,lk}|<Bn^{-1/2}\log n$ on $\mathcal{E}$.
From \eqref{medest2}, it also follows that for $A$ as in the proposition and $L_n$ defined in \eqref{cutoff}
\[ \max_{(l,k)\in \cT_{\text{int}}^*,\ l> \ell_0(n)} |\hat f_{\cT^*,lk}-f_{0,lk}| \leq \max_{(l,k),\ 2^{\ell_0(n)}<2^l < A2^{L_n}} |P_n\psi_{lk}-P_0\psi_{lk}|\eqqcolon Q_n.\]
We have that \[|P_n\psi_{lk}-P_0\psi_{lk}|\leq 2^{l/2}n^{-1}\left(|N(I_{l+1,2k})-nP_0(I_{l+1,2k})|+ |N(I_{l+1,2k+1})-nP_0(I_{l+1,2k+1})|\right).\] Therefore, on the event $\mathcal{B}_M$ from Lemma \ref{obs_event}, for some constant $C$ depending on $M,A,c_0$ and $\alpha$ only, and any $l$ as in the above supremum,
\begin{equation}\label{eq: bound empir. process}
|P_n\psi_{lk}-P_0\psi_{lk}| \leq C\sqrt{\frac{\log n}{n}}.
\end{equation} 
It follows that $Q\lesssim  n^{-1/2} \log n$ on the event $\mathcal{E}\cap\mathcal{B}_M$ that is such that $P_0\left(\mathcal{E}\cap\mathcal{B}_M\right)=1+o(1)$.
\end{proof}

\begin{lemma}\label{lemma: supnorm convergence median tree}
Let $\mathcal{T}^*$ as in \eqref{mediantree} and $\hat f_{\cT^*}$ as in \eqref{medest}. Then, for $f_0\in\mathcal{F}(\alpha, K,\mu)$,
\[ \| \hat f_{\cT^*}-f_0 \|_\infty =O_{P_0}\left(\left(\frac{\log^2{n}}{n}\right)^{\frac{\al}{2\al+1}}\right).\]
\end{lemma}

\begin{proof}
Let $\mathcal{E}$ as in Lemma \ref{median_tree_prop} and $\mathcal{B}_M$ as in Lemma \ref{obs_event}. On $\mathcal{E}\cap\mathcal{B}_M$, for $M$ large enough,
\begin{align*} \norm{f_0-f_{\mathcal{T}^*}}_\infty\leq& \sum_{l:\ 2^l< A2^{L_n}} 2^{l/2} \max\left[\underset{0\leq k<2^l,\ (l,k)\in\mathcal{T}_{\text{int}}^*}{\max} |\langle f_0-f_{\mathcal{T}^*}, \psi_{lk}\rangle|,\ \underset{0\leq k<2^l,\ (l,k)\notin\mathcal{T}_{\text{int}}^*}{\max} |\langle f_{0}, \psi_{lk}\rangle|\right] \\&+ \sum_{l:\ 2^l\geq A2^{L_n}} 2^{l/2} \underset{0\leq k<2^l}{\max} \left| \langle f_0, \psi_{lk}\rangle\right|,\end{align*}
using the usual inequality for densities $h,g$, $\norm{h-g}_\infty\leq \sum_{l\geq0}2^{l/2}\underset{0\leq k<2^l}{\max} \left| \langle h-g, \psi_{lk}\rangle\right|$. Since $f_0\in\Sigma(\alpha,K)$, the second term is smaller than $2^{-\alpha L_n}=O\left((n/\log n)^{-\alpha/(2\alpha+1)}\right)$ (up to a constant depending only on $\alpha$, $K$ and the constant $A$ from Lemma  \ref{median_tree_prop}). Then, the first term can itself be upper bounded by the sum of 
\[  \sum_{l:\ 2^l< A2^{L_n}} 2^{l/2} \underset{0\leq k<2^l,\ (l,k)\in\mathcal{T}_{\text{int}}^*}{\max} |\langle f_0-f_{\mathcal{T}^*}, \psi_{lk}\rangle|\lesssim 2^{L_n/2}\sqrt{\frac{\log n}{n}} =o\left(\Big(\frac{\log^2 n}{n}\Big)^{\alpha/ (2\alpha+1)}\right),\]
where we used that the argument of \ref{eq: bound empir. process} can be extended to $l\leq \ell_0(n)$ on $\mathcal{E}\cap\mathcal{B}_M$, and the term
\[ \sum_{l:\ 2^l<A2^{L_n}} 2^{l/2} \underset{0\leq k<2^l,\ (l,k)\notin\mathcal{T}_{\text{int}}^*}{\max} |\langle f_{0}, \psi_{lk}\rangle|.\]
It remains to upper bound this last quantity. Let's introduce \[L^*=\max\left\{l: \underset{0\leq k<2^l}{\max} |\langle f_{0}, \psi_{lk}\rangle|\geq Bn^{-1/2}\log n\right\}\] which is such that $2^{L^*}\asymp\left(\frac{n^{1/2}}{\log n}\right)^{1/(\alpha+1/2)}$ since $\underset{0\leq k<2^l}{\max} |\langle f_{0}, \psi_{lk}\rangle|\lesssim2^{-l(1/2+\alpha)}$. Then, on the event $\mathcal{E}$, the term in the above display is bounded by 
\begin{align*}
\sum_{l:\ 2^l< A2^{L_n}} 2^{l/2} \left(B\frac{\log n}{\sqrt{n}}\right)\wedge\underset{0\leq k<2^l}{\max} \left|\langle f_{0}, \psi_{lk}\rangle\right|&\leq\sum_{l:\ l\leq L^*} 2^{l/2} \left(B\frac{\log n}{\sqrt{n}}\right) +\\
&\qquad\qquad\sum_{l:\ 2^{L^*}<2^l< A2^{L_n}} 2^{l/2} \underset{0\leq k<2^l}{\max} |\langle f_{0}, \psi_{lk}\rangle|\\
&\lesssim \sqrt{2^{L^*}\frac{\log^2 n}{n}} + 2^{-\alpha L^*}\lesssim \left(\frac{\log^2 n}{n}\right)^{\alpha/ (2\alpha+1)}.
\end{align*}
Combining the previous bounds leads to, on $\mathcal{E}\cap\mathcal{B}_M$, \[\norm{f_0-f_{\mathcal{T}^*}}_\infty\leq C\left(\log^2 n/n\right)^{\alpha/ (2\alpha+1)}.\]

\end{proof}

\section{Nonparametric BvM theorem} \label{sec:npbvm}

%
%
%
%

\subsection{Space $\cM_0$ and limiting Gaussian process $\cN$}

Recall the definition of the space $\cM_0$ from \eqref{aimezero}, using an `admissible' sequence $w=(w_l)_{l\geq 0}$ such that $w_l/\sqrt{l} \to \infty$ as $l \to \infty$,
\begin{equation*}
\mathcal{M}_0=\mathcal{M}_0(w)=\left\lbrace x=(x_{lk})_{l,k}\text{ ; }\lim_{l \to \infty} \max_{0\leq k < 2^l} \frac{|x_{lk}|}{w_l}=0\right\rbrace.
\end{equation*}

\noindent Equipped with the norm $\di \|x\|_{\mathcal{M}_0}= \sup_{l\geq 0}\max_{0\leq k < 2^l} |x_{lk}|/w_l$, this is a separable Banach space. In a slight abuse of notation, we  write $f \in \mathcal{M}_0$ if the sequence of its Haar wavelet coefficients belongs to that space $(\langle f, \psi_{lk}\rangle)_{l,k}\in \mathcal{M}_0$ and for a process $(Z(f),\, f\in L^2)$, we write $Z\in\cM_0$ if the sequence $(Z(\psi_{lk}))_{l,k}$ belongs to $\cM_0(w)$ almost surely.  

{\em White bridge process.} For $P$ a probability distribution on $[0,1]$,  following \cite{cn14} one defines the $P$-white bridge process, denoted by $\mathbb{G}_P$, as  the centered Gaussian process indexed by the Hilbert space $\di L^2(P)=\{f:[0,1]\to \mathbb{R}; \int_0^1f^2dP<\infty\}$ with covariance
\begin{equation}\label{covv}
E[\mathbb{G}_P(f)\mathbb{G}_P(g)]=\int_0^1(f-\int_0^1fdP)(g-\int_0^1gdP)dP.
\end{equation} 
We  denote by $\mathcal{N}$ the law induced by $\mathbb{G}_{P_0}$ (with $P_0=P_{f_0}$) on $\cM_0(w)$. The sequence $(\mathbb{G}_P(\psi_{lk}))_{l,k}$  indeed defines a tight Borel Gaussian variable in $\cM_0(w)$, by Remark 1 of \cite{cn14}.

{\em Admissible sequences $(w_l)$.} 
The main purpose of the sequence $(w_l)$ is to ensure that $(\mathbb{G}_P(\psi_{lk}))_{l,k}$ belongs to $\mathcal{M}_0$. 
We refer to \cite{cn14}, Section 2.1 and Remark 1, for more background on the choice of $(w_l)$ in the present multiscale setting, and to \cite{cn13}, Section 1.2, for a similar discussion in an Hilbert space setting where the targeted loss is the $L^2$--norm.

To establish a nonparametric Bernstein--von Mises (BvM) result, following \cite{cn14} one first finds a space $\mathcal{M}_{0}$ large enough to have convergence at rate $\sqrt{n}$ of the posterior density to a Gaussian process. One can then derive results for some other spaces $\mathcal{F}$ using continuous mapping for continuous functionals $\psi : \mathcal{M}_0 \to \mathcal{F}$.

{\em Recentering the distribution.} To establish the BvM result, one also has to find a suitable way to center the posterior distribution. A possible centering is the median tree estimator $\hat f_{\cT^*}$ as in \eqref{medest}. Other centerings are possible, typically appropriately `smoothed' versions of the empirical measure $P_n$ associated to the sample $X_1,\ldots,X_n$
\begin{equation} \label{empmeas}
P_n=\frac{1}{n}\sum_{i=1}^n \delta_{X_i}.
\end{equation}
Let us now also note that another way to write the median tree estimator \eqref{medest} is
\begin{equation} \label{medest2}
 f_{\cT^*} = 1+\sum_{(l,k)\in \cT^*_{\text{int}}} (P_n\psi_{lk})\cdot \psi_{lk},
\end{equation}
where $P_n\psi_{lk}=n^{-1}\sum_{i=1}^n \psi_{lk}(X_i)$ are the empirical wavelet coefficients, and only terms corresponding to interior nodes $(l,k)$ in the median tree $\cT^*$ are active in the sum from the last display. From this we see that the median tree estimator \eqref{medest} can also be interpreted as a smoothed (or `truncated') version of the empirical measure $P_n$ in \eqref{empmeas}, with truncation occuring along the median tree $\cT^*$. Note also that if the prior $\Pi$ has flat initialisation up to level $l_0(n)$, then all nodes $(l,k)$ with $l\le l_0(n)$ are present in the above sum over $(l,k)\in\cT^*_{\text{int}}$. 

%

%
%
%
%

\subsection{Nonparametric BvM: statement} For the following result, we work with OPTs with flat initialisation as defined in Section \ref{cs-func}. This is discussed below the next statement.

We have the following Bernstein-von Mises phenomenon for $f_0$ in H\"older balls. For $C_n$ a function to be specified, we denote by $\tau_{C_n}$ the map $\tau_{C_n}:f\to \sqrt{n}(f-C_n)$.
\begin{theorem}\label{BVM}
Let $\mathcal{N}$ denote the distribution induced on $\cM_0(w)$ by the $P_0$--white bridge  $\mathbb{G}_{P_0}$ as defined in \eqref{covv} and let $C_n=\hat f_{\cT^*}$ the median tree estimator as in \eqref{medest}.  Let $\Pi$ be an OPT prior with flat initialisation with $l_0(n)$ that verifies $\sqrt{\log{n}}\le l_0(n)\le \log{n}/\log\log{n}$, and other than that for $l>l_0(n)$ with same parameters as the prior in Theorem \ref{contraction_rate}.  
Then 
 for every $\al\in(0,1]$, for $\mu>0$, $K\geq0$ and $\eta>0$,  
$$\di \sup_{f_0 \in \cF(\alpha,K,\mu)} E_{f_0}\left[\beta_{\mathcal{M}_0(w)}(\Pi(\cdot|X)\circ \tau_{C_n}^{-1},\mathcal{N})\right]\to 0,$$
as $n \to \infty$, for the admissible sequence $w_l=l^{2+\delta}$ for some $\delta>0$. 
\end{theorem}
\begin{remark}\label{rem-l} 
Recalling that the typical nonparametric cut--off sequence $\cL$ verifies $2^\cL\asymp n^{1/(1+2\al)}$, assuming $\ell_0(n)=o(\log{n})$ amounts to say that $\ell_0(n)$ does not `interfere' with the nonparametric cut-off $\cL$. Similar choices are made in \cite{ray17},  Corollary  3.6. Other choices of sequence $\ell_0(n)$ would also be possible, up to adjusting the sequence $(w_l)$ -- one can check that it suffices to have an increasing sequence $(w_l)$ such that $w_{l_0(n)}/\log{n} \to \infty$ (see, e.g. Theorem S--3 in the Supplement of \cite{cr21}) --; we do not consider these refinements here.  
\end{remark}
Theorem \ref{BVM}  states that the posterior limiting distribution is  Gaussian after rescaling; note that, similar to the first such result recently obtained in \cite{ray17}, one slightly modifies the OPT prior to fit the  first levels by assuming a flat initialisation. This is in fact necessary for the result to hold, as otherwise the posterior would not be tight at rate $1/\sqrt{n}$ in the space $\cM_0(w)$, as was noted in the white noise model in \cite{ray17}, Proposition 3.7. Let us also briefly comment on the recentering $C_n$: as follows from the proof of Theorem \ref{BVM}, one can replace $C_n=\hat f_{\cT^*}$ by another estimator that fits all first wavelet coefficients up to $\ell_0(n)$ and  such that $\|C_n-f_0\|_{\cM_0(\bar w)}=O_{P_0}(1/\sqrt{n})$, for $\bar w$ as in that proof, see also Remark \ref{rem-center} for more on this. 

\subsection{Nonparametric BvM: implications}
Using the methods of \cite{cn14}, this result  leads to several applications. A first direct implication (this follows from Theorem 5 in \cite{cn14}) is the derivation of a confidence set in $\cM_0(w)$. Setting
\begin{equation} \label{csmw}
\cD_n =\left\{f=(f_{lk}):\  \|f-C_n\|_{\cM_0(w)} \le \frac{R_n}{\rn} \right\},
\end{equation}
where $R_n$ is chosen in such a way that $\Pi[\cD_n\given X]=1-\ga$, for some $\ga>0$ (or taking the generalised quantile for the posterior radius if the equation has no solution) leads to a set $\cD_n$ with the following properties: it is a credible set by definition which is also asymptotically a confidence set in $\cM_0(w)$ and the rescaled radius $R_n$ is bounded in probability. 
Other applications are BvM theorems for functionals, as given a continuous map $\psi:\cM_0(w)\to \cE$ for some metric space $\cE$, convergence results in $\cM_0(w)$ can be translated into convergence in $\cE$ via the continuous mapping theorem, see \cite{cn14}. This is also at the basis of the proof of the Donsker Theorem \ref{thm-Donsker}.


\section{Proof of limiting shape results}
In this section we prove the nonparametric BvM Theorem \ref{BVM} and, as a fairly direct consequence given the results of \cite{cn14}, the Bayesian Donsker Theorem \ref{thm-Donsker}. 

\begin{proof}[Proof of Theorem \ref{BVM}]
The proof is similar to the corresponding proofs for P\'olya trees or spike--and--slab P\'olya trees, so we highlight only the few differences. The proof consists in two steps. First, proving convergence of finite--dimensional distributions and second, showing tightness of the rescaled posterior in a slightly smaller space.

Regarding convergence of finite--dimensional distributions, it suffices to note that for a  fixed depth $L>0$, the prior on wavelet coefficients of levels $l\le L$  (for large enough $n$ so that $\ell_0(n)>L$) coincides with the prior induced by a standard P\'olya tree, for which the convergence of finite--dimensional distributions is shown in \cite{c17}. 

Regarding tightness, let $\bar w=(\bar w_l)$ be the sequence $\bar w_l=w_l/l^{\delta/2}=l^{2+\delta/2}$. This sequence is increasing in $l$ and verifies $\bar w_l\geqa \sqrt{l}$,  $\bar w_l=o(w_l)$ as $l\to\infty$, and $\bar w_{\ell_0(n)}\ge \log{n}$, using the assumption on $\ell_0(n)$. Now by the same argument as in the proof of Theorem 3 in \cite{cm21}, to establish the nonparametric BvM it suffices to prove that  the distribution $\cL(\sqrt{n}(f-C_n)\given X)$ is tight in $\cM_0(\bar w)$, which is true if both laws  $\cL(\sqrt{n}(f-f_0)\given X)$ and $\cL(\sqrt{n}(f_0-C_n))$ are tight. 

Focusing first on the tightness of $\cL(\sqrt{n}(f-f_0)\given X)$, we wish to show that for any $\eta\in(0,1)$, one can find $M=M(\eta)$ large enough such that
\begin{equation} \label{tight}
 E_{f_0} \Pi[\|f-f_0\|_{\cM_0(\bar{w})}>M/\sqrt{n} \given X] \le \eta.
\end{equation}
We split, for $g=f-f_0$,
\[ \|g\|_{\cM_0(\bar{w})} \le \max_{l\le \ell_0(n), k} |g_{lk}|/{\bar{w}_l}+
 \max_{l> \ell_0(n), k} |g_{lk}|/{\bar{w}_l}=:(I)+(II).
 \]
For the term (I), as noted above, since the prior has a flat initialisation up to level $\ell_0(n)$, the induced prior and posterior on the first layers $l\le \ell_0(n)$ of wavelet coefficients coincide with the prior/posterior of a standard P\'olya tree, for which the corresponding tightness is proved in \cite{c17} (proof of Theorem 3). For the term (II), it follows from the proof of Theorem \ref{contraction_rate} (noting that the proof goes through with a prior with flat initialisation) that for $T_n$ as in that proof and given $l>\ell_0(n)$, for any $\cT\in T_n$ and on the event $\cB_M$,
\[ \int  \max_{k:\, (l,k)\in\cT_{int}} |f_{lk}-f_{0,lk}| d\Pi(f\given \cT,X)
\le C \sqrt{\frac{\log{n}}{n}} \]
and
\[ \max_{k:\, (l,k)\notin {\cT_{int}}} |f_{0,lk}| \le C\frac{\log{n}}{\sqrt{n}}. \]
Since $\bar w_{\ell_0(n)}\ge \log{n}$ as verified above, one deduces that for any $\cT\in T_n$ and on $\cB_M$ the term (II) above is $O(1/\sqrt{n})$. Putting pieces together what precedes implies, with $\cE=\{f_\cT,\ \cT\in T_n\}$ as in the proof of Theorem \ref{contraction_rate},
\[ \int_{\cE} \|f-f_0\|_{\cM_0(\bar w)} d\Pi(f\given X) =O_{P_0}(1/\sqrt{n}),\]
which in turn implies \eqref{tight} using $\Pi[\cE^c\given X]=o_{P_0}(1)$. 

It remains to prove tightness of $\cL(\sqrt{n}(f_0-C_n))$ in $\cM_0(\bar{w})$. Again, one splits along indices: for $l\le \ell_0(n)$, the posterior median tree estimator has same wavelet coefficients as the empirical measure $P_n$, and the estimate 
\[ E_{P_0} \max_{l\le \ell_0(n)} \max_k |\psg P_0-P_n,\psi_{lk} \psd| /{\bar w_l}\le C/\sqrt{n} \]
follows from the proof of Theorem 1 in \cite{cn14} (see equation (36) there and lines below). For $l>\ell_0(n)$, one  invokes the properties of the median tree estimator, namely
\begin{equation}\label{tr-tight}
\max_{l> \ell_0(n)} \max_k |\hat f_{\cT^*,lk}-f_{0,lk}| =O_{P_0}\left(\frac{\log{n}}{\sqrt{n}}\right),
\end{equation}
as in Lemma \ref{lemma: coeff med tree}, noting that the argument in that proof is unchanged for a prior with flat initialisation. This gives, using again $\bar w_{\ell_0(n)}\ge \log{n}$, that 
 \[ \max_{l\le \ell_0(n)} \max_k |\hat f_{\cT^*,lk}-f_{0,lk}| = O_{P_0}(1/\sqrt{n}),\]
 which gives the desired tightness property and concludes the proof. 
 \end{proof}

\begin{remark} \label{rem-center}
It follows from the proof of Theorem \ref{BVM} that there is quite some flexibility in the choice of the centering $C_n$. For instance, the projection $P_n(L_n)$ of the empirical measure $P_n$ onto the first $L_n$  levels of wavelet coefficients, with $L_n$ the oracle supremum--norm cut--off $(n/\log{n})^{1/(2\al+1)}$ can be used. This is because for $l\le \ell_0(n)$ the projection $P_n(L_n)$ has by definition same wavelet coefficients as the empirical measure $P_n$, while for $l>\ell_0(n)$ equation \eqref{tr-tight} holds for $\psg P_n(L_n), \psi_{lk}\psd$ instead of $f_{\cT^*, lk}$ (with the even better bound $O_{P_0}(\sqrt{\log{n}/n})$), as in the proof of Theorem 1 in \cite{cn14}.
\end{remark}

\begin{proof}[Proof of Theorem \ref{thm-Donsker}]
The results follows by applying Theorem 4 in \cite{cn14}: since the posterior distribution on $f$ satisfies the nonparametric BvM theorem \ref{BVM}, it suffices to check that the sequence $(w_l)$ satisfies the condition $\sum_l w_l2^{-l/2}<\infty$, which clearly holds,  and to note that the centering $C_n=f_{\cT^*}$ belongs to $L^2$. This shows that the Bayesian Donsker holds with centering $\hat F_n^{med}=\int_0^\cdot f_{\cT^*}$. By using remark \ref{rem-center}, the same result also holds with $\hat F_n^{med}$ replaced by the primitive, say $\mathbb{Z}_n(\cdot)$, of $P_n(L_n)$. But as noted in the proof of Corollary 1 in \cite{cn14} (see also Remark 9 in \cite{gn09}), we have $\|\mathbb{Z}_n-F_n\|_\infty=o_{P_0}(1/\sqrt{n})$, which implies the result with centering at $F_n$.
\end{proof}

\section{Miscellaneous}

We quickly remind that \[\Bar{Y}_\epsilon= E[Y_\epsilon\given X^{(n)}]=\frac{a+N_{X}(I_{\epsilon0})}{2a+N_X(I_\epsilon)}\] and we define $L_n$ as in \eqref{cutoff}.

\begin{lemma}
\label{obs_event}
Let $\alpha>0$, $K>0$ and $P_0$ be a distribution with a bounded density $f_0\in \Sigma(\alpha, K)$ w.r.t. Lebesgue density. Then, for any 
\[ M>\frac{1}{3}\left(\sqrt{\log 2}\sqrt{18\norm{f_0}_\infty+\log 2 }+\log 2\right),\]the event
\begin{align*}
&\mathcal{B}_M\coloneqq \Big\{\forall l\geq0, \quad \forall 0\leq k\leq 2^l-1, \\
&\qquad \qquad\qquad M^{-1}|N_X(I_{l,k})-nP_0(I_{l,k})|\leq \sqrt{\frac{n(l+L_n)}{2^l}}\vee (l+L_n)\eqqcolon M_{n,l}\Big\}
\end{align*}
is asymptotically certain under the law $P_0$ of the observations, i.e. 
\[P_0\left(\mathcal{B}_M^c\right)=o(1).\]
\end{lemma}
\begin{proof}
According to Bernstein's inequality, for any $l\geq0,\ 0\leq k\leq 2^l-1$,
\[ P_0\left(|N_X(I_{l,k})-nP_0(I_{l,k})|>MM_{n,l}\right)\leq 2\exp\left(-\frac{M^2M_{n,l}^2/2}{nP_0(I_{l,k})(1-P_0(I_{l,k}))+MM_{n,l}/3}\right).\]
By assumption, $P_0(I_{l,k})(1-P_0(I_{l,k}))\leq \norm{f_0}_\infty 2^{-l}$. Then, whenever $M_{n,l}=l+L_n$ (which is equivalent to $l+L_n\geq n2^{-l}$) or $M_{n,l}=\sqrt{\frac{n(l+L_n)}{2^l}}$, we can further upper bound the above quantity as
\[ P_0\left(|N_X(I_{l,k})-nP_0(I_{l,k})|>MM_{n,l}\right)\leq 2\exp\left(-\frac{M^2}{2\norm{f_0}_\infty+2M/3}(l+L_n)\right).\]
Therefore,
\[ P_0\left(\mathcal{B}_M^c\right) \leq 2\sum_{l\geq 0} 2^l\exp\left(-\frac{M^2}{2\norm{f_0}_\infty+2M/3}(l+L_n)\right) =O(2^{-L_n}) = O\left(\Big(\frac{\log n}{n}\Big)^{\frac{1}{2\alpha+1}}\right),\]
the latter equality being true whenever
\[ \frac{M^2}{2\norm{f_0}_\infty+2M/3}>\log 2,\]
i.e. $M>\frac{1}{3}\left(\sqrt{\log 2}\sqrt{18\norm{f_0}_\infty+\log 2 }+\log 2\right)$.
\end{proof}

\begin{lemma}
\label{beta control}
Suppose $f_0\in \Sigma(K, \alpha)$, with $0<\alpha\leq 1$. For $M'>0$, on the event $\mathcal{B}_M$ from Lemma \ref{obs_event}, the set

\[\mathcal{A}= \underset{\epsilon: |\epsilon| < L_n}{\cap} \left\{  |\Bar{Y}_{\epsilon0} - Y_{\epsilon0} |\leq M'\sqrt{\frac{L_n}{nP_0(I_{\epsilon0})}}  \right\} \]
is such that
\[\Pi[\mathcal{A}^c\given X] \lesssim \sum_{l\leq L_n} 2^le^{-{M'}^2\log n/4}\]
\end{lemma}

\begin{proof}
This proof comes from Lemmas $4$ and $5$ of \cite{cm21}. For completeness, we give here some details of the proof.
We have already explained that \[Y_{\epsilon0} \sim \text{Beta}(a+N_X(I_{\epsilon0}), a+N_X(I_{\epsilon0})).\]
We also noticed that on the event $\mathcal{B}_M$, $N_X(I_\epsilon)\to \infty$ uniformly for all $|\epsilon|\leq L_n$ for $n\to \infty$. Therefore, for $n$ sufficiently large, $a+N_X(I_{\epsilon0}) \wedge a+N_X(I_{\epsilon0}) \geq 8$ for $|\epsilon|<L_n$. Also, under our assumptions, Lemma [2] from \cite{c17} allows us to say that, for $n$ large enough, there exist $\mu, \nu$ such that

\[0<\mu\leq \frac{a+N_X(I_{\epsilon0})}{2a+N_X(I_{\epsilon0})+N_X(I_{\epsilon1})}\leq \nu<1\]
uniformly on all $|\epsilon|< L_n$. In addition, if $i=|\epsilon|$, we have that \[2a+N_X(I_{\epsilon0})+N_X(I_{\epsilon1}) \geq N_X(I_{\epsilon0}) \geq nP_0(I_{\epsilon0}) - M\sqrt{2nL_n2^{-i}}.\]Under our assumptions on $f_0$ and $L_n$, the last bound is itself lower bounded by $nP_0(I_{\epsilon0})/2$ for $n$ large enough. As a consequence, an application of Lemma 6 from \cite{c17} gives, for $x=M'L_n^{1/2}/2$, 

\[\Pi\left[|\Bar{Y}_{\epsilon0} - Y_{\epsilon0}| > \frac{x}{\sqrt{nP_0(I_{\epsilon0})}}\Big|\ X \right] \leq De^{-x^2/4}\] for some constant $D$. Finally, a union bound helps us to conclude that
\[\Pi[\mathcal{A}\given X] \lesssim \sum_{l\leq L_n} 2^l e^{-{M'}^2\log n/4}.\]
\end{proof}

\begin{lemma}[Theorem 1.5 of \cite{article_gamma}]
\label{bound_gamma}
For any $x>0$, 
\[ a\left(\frac{x+1/2}{e}\right)^{x+1/2} \leq \Gamma(x+1) \leq b\left(\frac{x+1/2}{e}\right)^{x+1/2},\]
where $\Gamma$ is usual Gamma function, and $a=\sqrt{2e}$ and $b=\sqrt{2\pi}$ are the best possible constants.
\end{lemma}

\end{document}